\documentclass[11pt]{amsart}
\usepackage[top=1.2in,bottom=1.2in,left=1.2in,right=1.2in]{geometry}
\usepackage[inline]{enumitem}

\usepackage{subcaption}

\usepackage{multicol}

\usepackage{amsthm}

\usepackage{mathtools}

\DeclarePairedDelimiter\floor{\lfloor}{\rfloor}

\usepackage{graphicx}
\usepackage{amssymb}
\usepackage{epstopdf}
\usepackage{hyperref}
\usepackage{xypic}

\newcommand\CC{\mathbb{C}}

\newcommand\FF{\mathbb{F}}

\newcommand\PP{\mathbb{P}}
\newcommand\QQ{\mathbb{Q}}
\newcommand\RR{\mathbb{R}}

\newcommand\ZZ{\mathbb{Z}}

\newcommand\cB{\mathcal{B}}

\newcommand\cF{\mathcal{F}}

\newcommand\cHH{\mathcal{H}}
\newcommand\cI{\mathcal{I}}

\newcommand\cK{\mathcal{K}}
\newcommand\cLL{\mathcal{L}}
\newcommand\cM{\mathcal{M}}

\newcommand\cQ{\mathcal{Q}}

\newcommand\cT{\mathcal{T}}

\newcommand\Sp{\text{Sp}}

\numberwithin{equation}{section}

\theoremstyle{plain}

\newtheorem{theorem}{Theorem}[section]
\newtheorem{proposition}[theorem]{Proposition}
\newtheorem{lemma}[theorem]{Lemma}

\newtheorem{corollary}[theorem]{Corollary}
\newtheorem{conjecture}[theorem]{Conjecture}
\newtheorem{question}{Question}

\theoremstyle{definition}

\newtheorem{remark}{Remark}[section]

\usepackage{paracol}

\usepackage{pinlabel}

\title{Picard groups of moduli spaces of curves with symmetry}
\author{Kevin Kordek}
\date{}

\begin{document}

\maketitle

\begin{abstract}
We study the Picard groups of moduli spaces of smooth complex projective curves that have a group of automorphisms with a prescribed topological action. One of our main tools is the theory of symmetric mapping class groups. In the first part of the paper, we show that, under mild restrictions, the moduli spaces of smooth curves with an abelian group of automorphisms of a fixed topological type have finitely generated Picard groups. In certain special cases, we are able to compute them exactly. In the second part of the paper, we show that finite abelian level covers of the hyperelliptic locus in the moduli space of smooth curves have finitely generated Picard groups. We also compute the Picard groups of the moduli spaces of hyperelliptic curves of compact type.
\end{abstract}

\section{Introduction}
Assume that $g\geq 2$ and let $S_g$ denote a closed orientable reference surface of genus $g$. The \emph{mapping class group} $\text{Mod}(S_g)$ is the group of orientation-preserving homeomorphisms of $S_g$ modulo isotopy. 

Let $H$ be a finite subgroup of $\text{Mod}(S_g)$. It follows from the solution of the Nielsen realization problem
 that there is a complex structure on $S_g$ such that $H$ lifts to an action on $S_g$ by automorphisms of this complex structure. In 
 the moduli space $\cM_g^H$ of genus $g$ curves with a group of automorphisms acting topologically like $H$ was constructed as the quotient of a contractible complex manifold by the \emph{symmetric mapping class group} $\text{Mod}_H(S_g)$, i.e. the normalizer of $H$ in $\text{Mod}(S_g)$. It follows from this construction that $\cM_g^H$ is naturally a \emph{quasiprojective orbifold}.

The moduli spaces $\cM_g^H$ generalize the moduli space $\cM_g$ of smooth curves of genus $g$, which can be recovered by taking $H = 1$. A familiar non-trivial example is the moduli space $\cHH_g$ of hyperelliptic curves; it can be constructed, as an orbifold, as the quotient of a contractible complex manifold by the \emph{hyperelliptic mapping class group}  $\text{Mod}_{\langle \sigma\rangle}(S_g)$ where $\sigma$ is the class of a hyperelliptic involution (see Figure \ref{hyperellipticinvolution} below).

\begin{figure}[t]
\centering
\includegraphics[scale = .35]{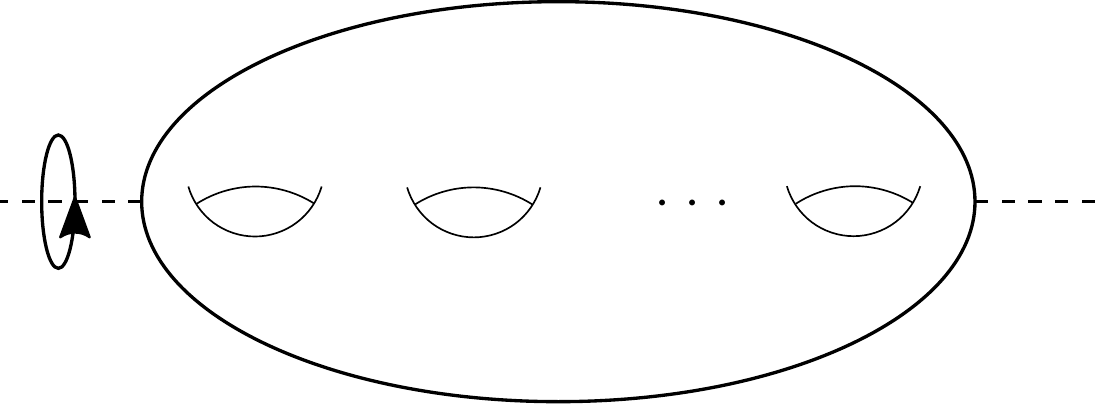}
\caption{\small Rotation about the indicated axis by $180^{o}$ gives a topological picture of a hyperelliptic involution.}
\label{hyperellipticinvolution}
\end{figure}

The principal aim of this paper is to use the structure of $\text{Mod}_H(S_g)$ to prove results about the geometry of $\cM_g^H$ and related moduli spaces of curves, especially the structure of their Picard groups. Our main results concern (1) moduli spaces of  smooth curves with abelian symmetries, i.e. the moduli spaces $\cM_g^H$ with $H$ abelian and (2) moduli spaces of smooth hyperelliptic curves with level structures and hyperelliptic curves of compact type.

\vspace{.1in}
Working entirely in the category of schemes, several authors have studied the enumerative geometry of various moduli stacks of curves with abelian symmetries (see, for example, \cite{arsievistoli}, \cite{pagani}, \cite{pomatalpotonini}). As far as the author can tell, the orbifold approach to this subject, in which symmetric mapping class groups play a central role, is novel. Where there is overlap, our results appear to agree. Although the main results of this paper concern groups of abelian symmetries, our approach allows us to prove some basic results concerning groups of nilpotent symmetries, which do not appear to have been dealt with in this context previously. We remark that our general setup can easily accommodate arbitrary groups of nonabelian symmetries.

\vspace{.1in}
We now give an overview of our results.

\subsection{Moduli of smooth curves with symmetry}
In the first part of this paper, we study the Picard group $\text{Pic}\ \cM_g^H$ of orbifold algebraic line bundles on $\cM_g^H$. Our approach is modeled after that of Hain \cite{hain94}, Putman \cite{putmanpicard} and Randal-Williams \cite{randalwilliams}, who studied the Picard groups of various moduli spaces of curves via the (orbifold) first Chern class homomorphism; in the present setting, this is a homomorphism
$$\text{Pic}\ \cM_g^H\rightarrow H^2(\text{Mod}_H(S_g),\ZZ)$$
which associates to a line bundle on $\cM_g^H$ its first Chern class. 

The features of this map are tightly controlled by the low-degree cohomology of $\text{Mod}_H(S_g)$. Therefore, much of the work in this paper consists in studying the structure of $\text{Mod}_H(S_g)$. To this end, we shall make heavy use of the \emph{Birman--Hilden} theory of symmetric mapping class groups, in which one views $\text{Mod}_H(S_g)$ as the subgroup of $\text{Mod}(S_g)$ consisting of the classes of those homeomorphisms that preserve the (branched) covering map $S_g\rightarrow S_g/H$. 

\vspace{.1in}
Our first main result concerns the structure of $\text{Pic}\ \cM_g^H$ when $H < \text{Mod}(S_g)$ is a finite abelian group. We begin by noting that, since $\text{Mod}_H(S_g)$ is finitely presentable, the Universal Coefficients Theorem implies that the torsion subgroup of $H^2(\text{Mod}_H(S_g), \ZZ)$ can be identified with the torsion subgroup of $H_1(\text{Mod}_H(S_g),\ZZ)$.

\vspace{.1in}
\begin{theorem}\label{maintheorem1}
Let $g\geq 2$ and suppose that $H < \text{Mod}(S_g)$ is a finite abelian group. Fix a complex structure on $S_g$ upon which $H$ acts by automorphisms.
\begin{enumerate}
\item \label{part1}  If $S_g/H\cong \PP^1$ then $\text{Pic}\ \cM_g^H$ is finite. If, additionally, $H_1(\text{Mod}_H(S_g),\hspace{.025in} \ZZ)$ is finite, then the first Chern class induces an isomorphism 
$$\text{Pic}\ \cM_g^H\xrightarrow{\cong} H_1(\text{Mod}_H(S_g),\ZZ) \subset H^2(\text{Mod}_H(S_g),\ZZ).$$ 
\item \label{part2}  If $S_g/H$ has genus at least $3$ then $H_1(\text{Mod}_H(S_g),\hspace{.025in} \ZZ)$ is finite, and $\text{Pic}\ \cM_g^H$ is finitely generated with torsion subgroup isomorphic to $H_1(\text{Mod}_H(S_g),\hspace{.025in} \ZZ)$.
\end{enumerate}
\end{theorem}
\vspace{.1in}


In Section \ref{nilpotentsymmetries} we will explain how to partially extend these results to finite nilpotent $H$ provided the covering map $S_g\rightarrow S_g/H$ is unbranched. 

\vspace{.1in}
An important test case is the moduli space $\cHH_g$ of hyperelliptic curves. Work of Birman--Hilden \cite{birmanhilden} implies that the abelianization of the hyperelliptic mapping class group is cyclic of order $4g+2$ when $g$ is even and order $8g+4$ when $g$ is odd. Part \ref{part1} of Theorem \ref{maintheorem1} then immediately implies the following.

\vspace{.1in}
\begin{corollary}\label{hyperellipticcorollary}
For all $g\geq 2$, we have 
\begin{equation*}
\emph{Pic}\ \cHH_g \cong\ \left\{
\begin{array}{lr}
\ZZ/(4g+2) \ZZ& g\ \text{even}\\
\ZZ/(8g+4) \ZZ& g\ \text{odd}
\end{array}
\right.
\end{equation*}
\end{corollary}

Work of Arsie--Vistoli  \cite{arsievistoli} shows that the Picard group of the moduli stack of smooth hyperelliptic curves over $\CC$ is cyclic of order $4g+2$ when $g$ is even and $8g+4$ when $g$ is odd. Corollary \ref{hyperellipticcorollary} is an analogue of their result in the setting of orbifolds. 
 
\subsection{Hyperelliptic curves}
In the second part of the paper, we focus on the geometry of the hyperelliptic loci in various moduli spaces of curves. Specifically, we shall consider two variants of the hyperelliptic locus in $\cM_g$, namely  
\begin{itemize}
\item The hyperelliptic loci in the finite abelian level covers of $\cM_g$
\item The moduli space of hyperelliptic curves of compact type.
\end{itemize}
We now summarize our results. 
\subsubsection{Hyperelliptic curves with level structures}
Let $\cM_g[m]$ denote the moduli space of smooth genus $g$ curves $C$ with a \emph{level $m$ structure}, i.e. a symplectic basis for $H_1(C,\ZZ/m\ZZ)$. The hyperelliptic locus $\cM_g^{hyp}[m]$  in $\cM_g[m]$ has a unique component when $m$ is odd and many (mutually isomorphic), disjoint components when $m$ is even. Let $\cHH_g[m]$ denote one of these components. 

\vspace{.1in}
By combining Hodge-theoretic techniques developed in \cite{hain94} with the results of \cite{brendlemargalitputman} on the structure of the \emph{hyperelliptic Torelli group}, which is the kernel of the canonical map $\text{Mod}_{\langle \sigma \rangle}(S_g)\rightarrow \Sp_g(\ZZ)$, we are able to prove the following result.
\vspace{.1in}

\begin{theorem}\label{hyperelliptictheorem}
For each $m\geq 1$, $\text{Pic}\ \cHH_g[m]$ is finitely generated.
\end{theorem}

\vspace{.1in}

We shall see that the Picard groups of both $\cHH_g$ and $\cHH_g[2]$ are finite and non-trivial. However, it is not yet clear whether $\text{Pic}\ \cHH_g[m]$ is finite or even if it is non-trivial when $m\geq 3$.

\subsubsection{Hyperelliptic curves of compact type}
\vspace{.1in}
A \emph{curve of compact type} is a stable nodal curve all of whose components are smooth and whose dual graph is a tree. The hyperelliptic locus $\cHH_g^c$ in the moduli space of genus $g$ curves of compact type is naturally a quasiprojective orbifold; we will construct it explicitly in Section \ref{hyperellipticsection} as the quotient of a simply-connected (but not contractible) complex manifold by an certain subgroup of the symplectic group $\Sp_g(\ZZ)$. Using this explicit description, we prove the following.

\vspace{.1in}

\begin{theorem}\label{picardhyperellipticcompacttype}
For each $g\geq 2$ we have
$\displaystyle \text{Pic}\ \cHH_g^c \cong \ZZ^{\floor{\frac{g}{2}}}\oplus
\left\{
\begin{array}{lr}
\ZZ/2\ZZ& g\ \text{even}\\
\ZZ/4\ZZ& g\ \text{odd}
\end{array}
\right.
$
\end{theorem}

\vspace{.1in}

The rank computation is geometric, and is readily carried out using known properties of the moduli space of stable hyperelliptic curves. On the other hand, our calculation of the torsion term relies on properties of the braid group and the hyperelliptic Torelli group. 

\subsection{Outline}
In Section \ref{preliminaries} we present background on some tools and ideas which will be used throughout the paper. In Section \ref{sectionsymmetricmappingclassgroups} we assemble the pieces needed to describe the quasiprojective orbifold structure on the moduli spaces $\cM_g^H$. In Section \ref{cohomologyofsphericalmappingclassgroupssection} we compute the rational cohomology of certain subgroups of the spherical mapping class groups for use in Section \ref{cycliccoversofp1section} where we deal with the moduli spaces $\cM_g^H$ for which $S_g/H \cong \PP^1$ and give the proof of the first part of Theorem \ref{maintheorem1}. We also exactly compute the Picard groups of certain moduli spaces of cyclic covers of $\PP^1$. In Section \ref{sectionhighergenus} we deal with the the moduli spaces for which $S_g/H$ has genus at least 3 and discuss a connection with an old conjecture of Ivanov. In Section \ref{curvesofcompacttypesection} we discuss moduli spaces of curves of compact type with level structure. Along the way, we compute the torsion subgroups of their Picard groups. In Section \ref{hyperellipticsection} we study moduli of smooth hyperelliptic curves with abelian level structures; we also study the moduli space of hyperellitiptic curves of compact type. It is in this last section that we prove Theorems \ref{hyperelliptictheorem} and \ref{picardhyperellipticcompacttype}. 

\subsection*{Acknowledgements} The author would like to thank Benson Farb, Dan Margalit, Gregory Pearlstein and Andrew Putman for helpful discussions. Special thanks are due to Tyrone Ghaswala and Becca Winarski for their correspondence and for explaining to the author their work on symmetric mapping class groups of cyclic covers of the sphere. The author would also like to thank Tatsunari Watanabe for many enlightening discussions regarding hyperelliptic mapping class groups. 

\section{Preliminaries}\label{preliminaries} 
\subsection{Mapping class groups}
Let $S^r_{g,n}$ denote a closed, orientable reference surface of genus $g$ with $n$ marked points and $r$ boundary components. We will occasionally find it convenient to think of the marked points as punctures. If either of the decorations $n,r$ is equal to zero, it will be omitted from the notation.  

The mapping class group $\text{Mod}(S^r_{g,n})$ is the group of orientation-preserving homeomorphisms of $S^r_{g,n}$ modulo isotopy. Here we require the homemorphisms to fix the marked points setwise and that the isotopies preserve the set of marked points. Homeomorphisms must restrict to the identity on the boundary. Isotopies must also be constant on the boundary. 

The pure mapping class group $\text{PMod}(S^r_{g,n})$ is the subgroup of $\text{Mod}(S^r_{g,n})$ consisting of those mapping classes that fix the marked points \emph{pointwise}. 

There is a short exact sequence
$$1\rightarrow \text{PMod}(S_{g,n})\rightarrow \text{Mod}(S_{g,n})\rightarrow \mathfrak{S}_n\rightarrow 1$$
where $\mathfrak{S}_n$ is the symmetric group on $n$ letters; here we think of $\mathfrak{S}_n$ as the symmetric group on the $n$ marked points. 

\vspace{.1in}
There is a natural action of $\text{Mod}(S_g)$ on $H_1(S_g,\ZZ)$ that preserves the (symplectic) intersection form. Consequently, we have a homomorphism $\text{Mod}(S_g)\rightarrow \Sp_g(\ZZ)$ to the group of $2g\times 2g$ integral symplectic matrices. It is a classical fact that this map is surjective. For all $n\geq 1$, there is a natural surjective homomorphism $\text{PMod}(S_{g,n})\rightarrow \text{Mod}(S_g)$ obtained by forgetting the marked points. 

\vspace{.1in}
The \emph{Torelli group} $\cI(S_{g,n})$ is the kernel of the composition 
$$\text{PMod}(S_{g,n})\rightarrow \text{Mod}(S_g)\rightarrow \Sp_g(\ZZ).$$ 
It is a torsion-free group.

The \emph{Johnson subgroup} $\cK(S_{g,n})$ of $\cI(S_{g,n})$ is the subgroup generated by all Dehn twists on simple closed curves that separate $S_{g,n}$ into two closed subsurfaces (separating simple closed curves). 

\vspace{.1in}
Let $\Sp_g(\ZZ)$ denote the group of $2g\times 2g$ symplectic matrices with integer entries. For each $m\geq 1$, define the level $m$ subgroup $\Sp_g(\ZZ)[m]$ of $\Sp_g(\ZZ)$ by 
$$\Sp_g(\ZZ)[m] = \left\{M\in \Sp_g(\ZZ)\ |\ M = I_{2g\times 2g}\ \ \text{mod}\ m\right\}.$$
Alternatively, $\Sp_g(\ZZ)[m]$ is the kernel of the canonical homomorphism $\Sp_g(\ZZ)\rightarrow \Sp_g(\ZZ/m\ZZ)$. It is torsion-free provided $m\geq 3$. 

We define the level $m$ subgroup $\text{PMod}(S_{g,n})[m]$ of $\text{PMod}(S_{g,n})$ to be the kernel of the composition $\text{PMod}(S_{g,n})\rightarrow \Sp_g(\ZZ)\rightarrow \Sp_g(\ZZ/m\ZZ)$. It is torsion-free provided $m\geq 3$, a consequence of the fact that $\cI(S_{g,n})$ is torsion-free \cite{hain94}.
\subsection{Group cohomology}
Let $G$ be a group, $R$ a commutative ring, and $M$ an $RG$-module. The module $M^G$ of invariants is the submodule of $M$ spanned by all elements that are fixed by the $G$ action. The module $M_G$ of coinvariants is the maximal $RG$-module quotient on which $G$ acts trivially.  

\vspace{.1in}
Assume that $H$ is a normal subgroup with quotient $Q = G/H$. 

\subsubsection{The 5-term exact sequence}
For any commutative ring $R$, the Hochschild-Serre spectral sequence computes the cohomology of $G$ from the cohomology of $H$ and $Q$. This is a first quadrant spectral sequence given by $$E_2^{p,q} = H^p(Q, H^q(H,R))\implies H^{p+q}(G,R).$$

We will occasionally need to make use of the associated 5-term exact sequence 
$$0\rightarrow H^1(Q,R)\rightarrow H^1(G,R)\rightarrow H^1(H,R)^{Q}\rightarrow H^2(Q,R)\rightarrow H^2(G,R)$$
that comes from this spectral sequence.

\vspace{.1in}
There is a dual version for group homology.  For more details, see \cite{brown}, for example.

\subsubsection{The transfer}
If $H$ is any finite-index subgroup of $G$ (not necessarily a normal subgroup) then for each $k\geq 0$ the inclusion $H\rightarrow G$ induces an injective map $H^k(G,\QQ)\rightarrow H^k(H,\QQ)$. If $H$ is, in addition, normal in $G$, this induces an isomorphism 
$$H^k(G,\QQ)\rightarrow H^k(H,\QQ)^{G/H}.$$
Dually, under these conditions the natural map $H_k(H,\QQ)\rightarrow H_k(G,\QQ)$ is surjective, and induces an isomorphism $H_k(H,\QQ)_{G/H}\rightarrow H_k(G,\QQ)$.

\subsection{Orbifolds}

In this paper, by an \emph{orbifold} we mean a pair $(X,G)$ where
\begin{enumerate}
\item  $X$ is a simply-connected complex manifold.
\item  $G$ is a discrete group acting properly discontinuously $X$ via biholomorphisms such that there is a finite-index subgroup $G' < G$ whose action on $X$ is free.
\end{enumerate}

There is a special class of orbifolds suited for use in algebraic geometry. These are known as \emph{quasiprojective orbifolds}. In this section we present general background material on quasiprojective orbifolds and their Picard groups. More details can be found in \cite{hain2000} or \cite{putmanpicard}, for example. 

\subsubsection{Quasiprojective orbifolds }
Let $X$ be a complex manifold and $G$ a discrete group acting on $X$ biholomorphically on the left. Assume that the action of $G$ on $X$ is virtually free, i.e. $G$ has a subgroup $G'$ of finite index that acts freely. Then $G'\backslash X$ is a complex manifold and $G\backslash X$ is naturally a normal analytic space by a theorem of Cartan \cite[p.257]{acgh}. 

\vspace{.1in}
Assume now that $G\backslash X$ is biholomorphic to a normal complex quasiprojective variety. Once the algebraic structure on $G\backslash X$ is fixed, the generalized Riemann Existence Theorem \cite{mumford} implies that the complex manifold $G'\backslash X$ admits a unique quasiprojective algebraic structure such that the natural map $G'\backslash X\rightarrow G\backslash X$ is a morphism of algebraic varieties. 
With respect to this algebraic structure, $G/G'$ acts on $G'\backslash X$ by automorphisms.

If, in addition, $X$ is simply-connected, the pair $(X,G)$ is a \emph{quasiprojective orbifold} and $(X,G')$ is a \emph{quasiprojective finite cover} of $(X,G)$. We define a \emph{regular quasiprojective finite cover} of $(X,G)$ to be a quasiprojective finite cover $(X,G')$ with $G'$ a normal subgroup of $G$. 

There is a natural sense in which $(X,G')$ can be identified with $G'\backslash X$. If $G' $ is a normal subgroup of $G$, then $(X,G)$ can be thought of as the quotient of $G'\backslash X$, in an orbifold sense, by the finite group $G/G'$.

\subsubsection{The Picard group of a quasiprojective orbifold}

Let $\Gamma$ be a finite group and suppose that $Y$ is a variety with a $\Gamma$-action. A $\Gamma$-equivariant (algebraic) line bundle on $Y$ is an algebraic line bundle $\pi: \cLL\rightarrow Y$ with a $\Gamma$-action such that 
\begin{enumerate}
\item $\pi$ is equivariant with respect to this action and 
\item  the $\Gamma$-action is linear on the fibers of $\pi$. 
\end{enumerate}

Two $\Gamma$-equivariant line bundles on $G'\backslash X$ are said to be isomorphic if there is a $\Gamma$-equivariant isomorphism between them.

\vspace{.1in}
Now let $(X,G)$ be a quasiprojective orbifold and suppose that $(X,G')$ is a regular quasiprojective finite cover. Let $\Gamma = G/G'$. The Picard group $\text{Pic}(X,G)$ of $(X,G)$ is the group of isomorphism classes of $\Gamma$-equivariant algebraic line bundles on the smooth variety $G'\backslash X$, where the group operation is induced by tensor product of line bundles. It is readily checked that this definition does not depend upon the choice of regular quasiprojective finite cover $(X,G')$. 

\subsubsection{Chern classes} 
For a complex algebraic variety $X$, the first Chern class is a homomorphism $c_1: \text{Pic}\ X\rightarrow H^2(X,\ZZ)$. By replacing ordinary cohomology with equivariant cohomology, this notion can be adapted to the orbifold setting. 

\vspace{.1in}
We begin this section by reviewing some basic facts about equivariant cohomology. We then discuss some general properties of Chern class homomorphism on a quasiprojective orbifold. 

\vspace{.1in}
For a topological space $X$ and a discrete group $G$ acting properly discontinuously on $X$, the homotopy quotient of $X$ with respect to $G$ is the quotient space
$$E\Gamma\times_{G} X= G \backslash\left(EG\times X\right)$$
where $EG$ denotes the universal cover of the classifying space of $G$ (on which $G$ acts freely) and $G$ acts on the product diagonally. The $k$th equivariant cohomology group of $X$ is defined to be 
$$H^k_{G}(X,\ZZ): = H^k(EG \times_{G} X,\ZZ).$$

\vspace{.1in}
The following result gives conditions under which $H^k_G(X,\ZZ) \cong H^k(G,\ZZ)$. 
\begin{proposition}[see, for example, \cite{putmanpicard}]\label{equivariantcohomologyprop}
Assume that $G$ is a discrete group acting properly discontinuously and cellularly on a connected CW-complex $X$. If $H$ is a normal subgroup of $G$ acting freely, then $H^k_{G/H}(H\backslash X,\ZZ)\cong H^k_G(X,\ZZ)$ for all $k\geq 0$. If $X$ is $n$-connected. Then $H^k_G(X,\ZZ)\cong H^k(G,\ZZ)$ for all $0\leq k\leq n$. 
\end{proposition}
Now assume that $(X,G)$ is a quasiprojective orbifold and fix a regular quasiprojective finite cover $(X,G')$. Let $\Gamma = G/G'$. There is an equivariant first Chern class mapping 
$$c_1: \text{Pic}(X,G)\rightarrow H^2_{\Gamma}(G'\backslash X,\ZZ)$$ 
that sends the class of a $\Gamma$-equivariant algebraic line bundle on $G'\backslash X$ to its equivariant Chern class. Concretely, if $\cLL$ is an equivariant algebraic line bundle on $G'\backslash X$, its equivariant first Chern class is the (ordinary) first Chern class of the (topological) line bundle 
$$ E\Gamma \times_{\Gamma} \cLL \rightarrow E\Gamma \times_{\Gamma} G'\backslash X.$$
We define $\text{Pic}^0(X,G)$ to be the kernel of $c_1$.

\vspace{.1in} Let $Y$ be a smooth quasiprojective variety. Standard arguments from Hodge theory can be used to prove that the $\text{Pic}^0(Y) = 0$ if and only if  $H^1(Y,\ZZ) = 0$. The analogous statement is true for orbifolds. 
\begin{theorem}[Hain \cite{hain2000}]\label{hainpic0theorem}
Let $(X,G)$ be a quasiprojective orbifold and $(X,G')$ a regular quasiprojective finite cover. If $H^1_{G/G'}(G'\backslash X,\ZZ) = 0$, then $\text{Pic}^0(X,G) = 0$. 
\end{theorem}
Additionally,  the image of $c_1: \text{Pic}(Y)\rightarrow H^2(Y,\ZZ)$ contains the entire torsion subgroup of $H^2(Y,\ZZ)$, which is isomorphic to the torsion subgroup of $H_1(Y,\ZZ)$ by the Universal Coefficients Theorem. These torsion classes are exactly the Chern classes of flat line bundles on $Y$. This result also carries over to the orbifold setting.
\begin{theorem}[Putman \cite{putmanpicard}]\label{putmanpicardtheorem}
Suppose that $(X,G)$ is a quasiprojective orbifold with regular quasiprojective finite cover $(X,G')$. The image of $c_1: \text{Pic}(X,G)\rightarrow H^2_{\Gamma}(G'\backslash X,\ZZ)$ contains the entire torsion subgroup.
\end{theorem}

By the Universal Coefficients Theorem and the fact that $X$ is simply-connected, the torsion subgroup of $H^2_{\Gamma}(G'\backslash X,\ZZ)$ is isomorphic to the torsion subgroup of $H_1(G,\ZZ)$.

\subsubsection{Finite generation of Picard groups} Suppose that $X$ is a smooth quasiprojective variety. Then $\text{Pic}\ X$ is the extension of a finitely generated group by a complex torus. It follows that $\text{Pic}\ X$ is finitely generated if and only if $\text{Pic}\ X \otimes \QQ$ is finite-dimensional. The following generalization appears to be well known. However, being unable to find a reference, we give a proof here. 

\begin{lemma}
Assume that $X$ is the quotient of a smooth quasiprojective variety $Y$ by a finite group. Then $\text{Pic}\ X$ is finitely generated if and only if $\text{Pic}\ X\otimes \QQ$ is finite-dimensional.
\end{lemma}
\begin{proof}
Assume that $X = G\backslash Y$ with $G$ finite. Without loss of generality, we shall assume that $G$ acts effectively.  Let $Z\subset Y$ denote the locus of points fixed by some element of $G$. Set $Y' = Y\setminus Z$ and define $X' = G\backslash Y'$. 

For a variety $W$, let $A^1(W)$ denote the Chow group of divisors on $W$ modulo rational equivalence. Then the kernel of the canonical map $A^1(X)\rightarrow A^1(X')$ is generated by the codimension 1 components of $Z$. Since $X$ has finite quotient singularities, Weil divisors on $X$ are $\QQ$-Cartier and we have $\text{Pic}\ X\otimes \QQ = A^1(X)\otimes \QQ$. If $\text{Pic}\ X\otimes \QQ$ is finite-dimensional, then so is $A^1(X')\otimes \QQ = \text{Pic}\ X'\otimes \QQ$. Since $X'$ is smooth and quasiprojective, $\text{Pic}\ X' =A^1(X')$ is finitely generated. Thus $A^1(X)$ is finitely generated, as it is an extension of a finitely generated abelian group by a finitely generated abelian group. Since $X$ is normal, $\text{Pic}(X)$ is a subgroup of $A^1(X)$. Thus $\text{Pic}(X)$ is finitely generated. 
\end{proof}

By \cite{knopkraftvust}, if $G$ is a finite group acting on a smooth quasiprojective variety $Y$, then $\text{Pic}\ G\backslash Y$ is a finite-index subgroup of the group $\text{Pic}_G Y$ of $G$-equivariant line bundles on $Y$. From this, we immediately deduce the following result.
 
 \vspace{.1in}
\begin{lemma}\label{picardlemma}
Let $(X,G)$ be a quasiprojective orbifold. Then $\emph{Pic}(G\backslash X)\otimes \QQ$ is finite-dimensional if and only if $\emph{Pic}(X,G)$ is finitely generated.
\end{lemma}

%
%
\section{Symmetric mapping class groups and moduli spaces of symmetric curves}\label{sectionsymmetricmappingclassgroups}
Suppose that $g\geq 2$, and that $p: S_g\rightarrow X$ is a regular covering branched over a finite set $\textbf{P}\subset X$ of $n\geq 0$ points. Let $H$ denote the group of deck transformations of $p$. Since $p$ is regular, there is natural identification $X\cong S_g/H$. Note that the restriction $S_g\setminus p^{-1}(B)\rightarrow X\setminus B$ is a regular unbranched covering.

\vspace{.1in}
A homeomorphism $f$ of $S_g$ is said to be \emph{symmetric} with respect to $p$ if $f$ maps fibers of $p$ to fibers of $p$. The \emph{symmetric mapping class group} $\text{SMod}_p(S_g)$ is the group of isotopy classes of orientation-preserving homeomorphisms of $S_g$ that are symmetric with respect to $p$. Note that we do not require that the isotopies be fiber-preserving. It is clear that $H < \text{SMod}_p(S_g)$. 

Standard arguments in Teichm\"{u}ller theory imply that $\text{Mod}_H(S_g) = \text{SMod}_p(S_g)$ as subgroups of $\text{Mod}(S_g)$ (c.f. \cite{birmanhilden}). We will use these notations interchangeably when there is no danger of confusion. 
\subsection{Liftable mapping class groups and the Birman--Hilden property}

Assume that $p: S_g\rightarrow X$ is a regular finite-sheeted branched covering. We shall say that a homeomorphism $f$ of $X$ lifts to a homeomorphism of $S_g$ if there exists a homeomorphism $\tilde{f}$ of $S_g$ such that $f\circ p = p\circ \tilde{f}$. 

\subsubsection{Liftable mapping class groups}
In what follows, we will regard $X$ as a surface with marked points given by the set $\textbf{P}$ of branch points of $p$. We define the \emph{liftable mapping class group} $\text{LMod}_p(X)$ to be the subgroup of $\text{Mod}(X)$ consisting of isotopy classes of orientation-preserving homeomorphisms of $(X, \textbf{P})$ that lift to $S_g$. 

\subsubsection{Birman--Hilden theory}
There is a natural homomorphism 
$$\Phi: \text{LMod}_p(X)\rightarrow \text{SMod}_p(S_g)/H$$ obtained by sending the class of a homeomorphism $[f]\in \text{LMod}_p(X)$ to the class of a lift $[\tilde{f}]$. A celebrated result of Birman--Hilden \cite{birmanhilden} states that, under the conditions we have imposed on $p$, the homomorphism $\Phi$ is actually an isomorphism. Coverings with this property are said to have the \emph{Birman--Hilden property}. We therefore have a short exact sequence
\begin{equation}\label{birmanhildenextension}
1\rightarrow H\rightarrow \text{SMod}_p(S_g)\rightarrow \text{LMod}_p(X)\rightarrow 1
\end{equation}
where the homomorphism $\text{SMod}_p(S_g)\rightarrow \text{LMod}_p(X)$ sends the class of a homeomorphism to the class $[\overline{g}]$ of the induced homeomorphism of $X$.

\subsubsection{Finiteness properties of symmetric mapping class groups}
It follows from covering space theory that the liftable mapping class group $\text{LMod}_p(X)$ is a finite-index subgroup of $\text{Mod}(X)$. Since $\text{Mod}(X)$ is finitely presented, so too is $\text{LMod}_p(X)$. 

The fact that the deck group $H$ is finite, along with (\ref{birmanhildenextension}), implies that $\text{SMod}_p(S_g)$ is also finitely presented. Thus, for all $g\geq 2$ and all finite covers $p:S_g\rightarrow S_h$, the symmetric mapping class group $\text{SMod}_p(S_g)$ is finitely presented. This implies that the integral homology and cohomology of $\text{SMod}_p(S_g)$ are finitely generated in degrees 1 and 2.
\subsection{The Teichm\"{u}ller description of the moduli space}
In this section we review some basic Teichm\"{u}ller theory, which we shall use to describe the orbifold structure on $\cM_g^{H}$.

\subsubsection{Moduli spaces of pointed curves}\ 
Assume that $2g-2+n > 0$. Let $S_{g,n}$ denote a closed reference surface of genus $g$ with $n$ marked points. Recall that the Teichm\"{u}ller space $\mathfrak{X}_{g,n}$ of $S_{g,n}$ is the space of all complex structures on $S_{g,n}$ up to isotopy. Points of $\mathfrak{X}_{g,n}$ are equivalence classes of pairs $((C,\textbf{P}),\varphi)$ where $C$ is a smooth complete algebraic curve of genus $g$, $\textbf{P}$ is a set of $n$ marked points $\varphi: (C,\textbf{P})\rightarrow S_{g,n}$ is a homeomorphism. The mapping class group $\text{Mod}(S_{g,n})$ acts via biholomorphisms on $\mathfrak{X}_{g,n}$.

\vspace{.1in}
The moduli space of smooth genus $g$ curves with $n$ \emph{ordered} marked points $\cM_{g,n}$ can be realized as the orbifold $\left(\text{PMod}(S_{g,n}), \mathfrak{X}_{g,n}\right)$.  It is a quasiprojective orbifold because $\text{PMod}(S_{g,n})[m]$ is torsion-free if $m\geq 3$, and the quasiprojective orbifold 
$$\cM_{g,n}[m] = (\text{PMod}(S_{g,n})[m], \mathfrak{X}_{g,n})$$ is a regular quasiprojective finite cover. 

\vspace{.1in}
The moduli space $\cM_{g,[n]}$ of smooth curves of genus $g$ with $n$ \emph{unordered} points is the orbifold $\left(\text{Mod}(S_{g,n}), \mathfrak{X}_{g,n}\right)$. It is a quasiprojective orbifold, and there is a Galois covering $\cM_{g,n}\rightarrow \cM_{g,[n]}$ with Galois group $\mathfrak{S}_n$, which acts by permuting the marked points.

\subsubsection{Moduli of symmetric curves} Let $H$ be a finite subgroup of $\text{Mod}(S_g)$. 
It follows from Kerckhoff's solution of the Nielsen realization problem \cite{kerckhoff} that the action of $H$ on $\mathfrak{X}_g$ has a nonempty set of fixed points \cite{gonzalezdiezharvey}. It can be shown that the locus $\mathfrak{X}_g^{H}$ of points that are fixed by the $H$-action form a contractible complex submanifold of $\mathfrak{X}_g$ of dimension $3g'-g'+n$, where $g'$ is the genus of the quotient $X = S_g/H$ and $n$ is the number of branch points of the covering $S_g\rightarrow X$. In fact, $\mathfrak{X}_g^{H}$ is itself isomorphic to a Teichm\"{u}ller space. It is known \cite{gonzalezdiezharvey} that the stabilizer of $\mathfrak{X}_g^{H}$ in $\text{Mod}(S_g)$ is equal to the normalizer $\text{Mod}_{H}(S_g)$ of $H$ in $\text{Mod}(S_g)$. 
\begin{theorem}[Gonz\'{a}lez-D\'{i}ez--Harvey \cite{gonzalezdiezharvey}]\label{teichmullerspaces}
Assume that $g\geq 2$, and suppose that the covering $p: S_g\rightarrow S_g/H$ has exactly $n$ branch points. Then there is a biholomorphism $\phi: \mathfrak{X}_g^{H}\rightarrow \mathfrak{X}_{g',n}$ where $g'$ is the genus of $S_g/H$. 
\end{theorem}
The map $\phi$ in Theorem \ref{teichmullerspaces} can be described as follows. Let $[C, \varphi]\in \mathfrak{X}_g$ denote a point in Teichm\"{u}ller space, where $\varphi: C\rightarrow S_g$ is a marking. Let $G <\text{Aut}(C)$ denote the subgroup of automorphisms corresponding to the topological action of $H$ on $S_g$. Let $\overline{C} = G\backslash C$, and let $\textbf{P}$ denote the set of branch points of the ramified covering $C\rightarrow \overline{C}$. Then $\varphi$ descends to a marking $\overline{\varphi}$ of the \emph{pointed} curve $(\overline{C}, \textbf{P})$, and we have $\phi [C,\varphi] = [\overline{C}, \overline{\varphi}]$.

\vspace{.1in}
As a set, the moduli space $\cM_g^H$ of curves with a group of automorphisms acting topologically like $H$ consists equivalence classes of pairs $(C, A)$ where $C$ is a smooth curve of genus $g$ and $A < \text{Aut}(C)$ with the following property: there exists a homeomorphism $f: C\rightarrow S_g$ such that $fAf^{-1} = H$. Two such pairs $(C_1, A_1)$ and $(C_2,A_2)$ are equivalent if and only if there exists an isomorphism $\alpha: C_1\rightarrow C_2$ such that $\alpha A_1 \alpha^{-1} = A_2$. By \cite{gonzalezdiezharvey}, $\cM_g^{H}$ can be constructed analytically as the quotient space 
$$\text{Mod}_{H}(S_g)\backslash \mathfrak{X}_g^{H}.$$ 

\vspace{.1in}
Notice that the subgroup $H < \text{Mod}_{H}(S_g) $ acts trivially on $\mathfrak{X}_g^{H}$ by definition.
Since $\text{SMod}_p(S_g) = \text{Mod}_{H}(S_g)$, we have $\text{Mod}_{H}(S_g)/H = \text{LMod}_p(X, \textbf{P})$ and there is an isomorphism of analytic spaces
$$\cM_g^{H}\cong \text{LMod}_p(X,\textbf{P})\backslash \mathfrak{X}_{g',n}.$$ 
The fact that $\text{LMod}_p(X,\textbf{P})$ is a finite-index subgroup of $\text{Mod}(X, \textbf{P})$ implies, then, that the natural holomorphic map 
$$\text{LMod}_p(X,\textbf{P})\backslash \mathfrak{X}_{g',n}\rightarrow \cM_{g',[n]}.$$
is finite in the sense that it is proper, surjective, and has finite fibers. Since $\cM_{g',[n]}$ has a natural quasiprojective algebraic structure, the generalized Riemann Existence Theorem guarantees the existence of a unique (quasiprojective) algebraic structure on $\cM_g^{H}$ so that the map 
$$\cM_g^{H}\xrightarrow{\cong} \text{LMod}_p(X,\textbf{P})\backslash \mathfrak{X}_{g',n}\rightarrow \cM_{g',[n]}$$ 
is a morphism of complex algebraic varieties. Moreover, $\text{Mod}_{H}(S_g)$ admits a torsion-free subgroup of finite index, namely the intersection $\text{Mod}(S_g)[3]\cap \text{Mod}_{H}(S_g)$. 

\vspace{.1in}
We shall regard $\cM_g^{H}$ as the quasiprojective orbifold $\left(\text{Mod}_H(S_g), \mathfrak{X}^H_g\right)$ from now on.

\begin{remark}
The orbifold $(\text{SMod}_p(S_g)/H, \mathfrak{X}_g)$ is also quasiprojective and has the same underlying analytic space as $\cM_g^{H}$. However, it does not encode the fact that the objects we wish to parametrize all have a group of automorphisms acting topologically like $H$. With the orbifold structure we have placed on $\cM_g^{H}$, there is a natural orbifold mapping $\cM_g^{H}\rightarrow \cM_g$ which can be thought of as sending the (symmetric) isomorphism class of a symmetric curve to the (non-symmetric) isomorphism class of the underlying curve. 
\end{remark}
%
%
\section{Cohomology of spherical mapping class groups}\label{cohomologyofsphericalmappingclassgroupssection}
In this section, we prove some results on the structure of certain finite-index subroups of $\text{Mod}(S_{0,n})$ which will be of use to us in later sections. 

\vspace{.1in}
For $n\geq 1$, let $D_n$ denote a closed disk with $n$ marked points situated in its interior. Without loss of generality, we assume that $D_n \subset \RR^2$ and that the marked points all lie on a straight line. The braid group $B_n$ on $n$ strands is the mapping class group $\text{Mod}(D_n)$ (here we require all homeomorphisms and isotopies to fix the boundary pointwise). 

\vspace{.1in}
For each $n$, there is a homomorphism $B_n\rightarrow \text{Mod}(S_{0,n+1})$ from the braid group $B_n$ on $n$ strands obtained by capping the boundary with a once-marked disk. We define 
$$M_n = \text{Im} \left(B_n \rightarrow \text{Mod}(S_{0,n+1})\right).$$

\noindent This group has index $n+1$ in $\text{Mod}(S_{0,n+1})$.

 \vspace{.1in} We are aiming to prove the following.
\begin{proposition}\label{section4mainprop}
Assume that $g\geq 2$ and that $p: S_g\rightarrow S_0$ is a regular finite-sheeted covering of the sphere branched over a set of $n$ points. If the image of the natural homomorphism $\text{SMod}_p(S_g)\rightarrow \text{Mod}(S_{0,n})$ contains $M_{n-1}$, then for all $j\geq 1$ we have
$$H^j(\text{SMod}_p(S_g),\QQ) = 0.$$
\end{proposition}

First, we collect a few facts about the braid groups.
\subsection{Braid groups} 
The braid group $B_n$ has a well-known presentation, the \emph{Artin presentation}, consisting of $n-1$ generators $\sigma_1,\ldots, \sigma_{n-1}$ subject to the relations
\begin{equation}
\left\{
\begin{array}{lr}
\sigma_i\sigma_{i+1}\sigma_i = \sigma_{i+1}\sigma_i\sigma_{i+1}&\ \text{for all}\ i \\
\sigma_i\sigma_j = \sigma_j\sigma_i&\ \text{for}\ |i-j| > 1 \\
\end{array}
\right.
\end{equation}
The generator $\sigma_i$ can be viewed as the half-twist along the line segment in $D_n$ joining the $i$th marked point to the $(i+1)$st marked point.

The kernel $T$ of the homomorphism $B_n\rightarrow M_n$ is the infinite cyclic group generated by isotopy class of the Dehn twist on the boundary of $D_n$. In terms of the Artin generators the class of this twist is equal to
$$(\sigma_1\cdots \sigma_{n-1})^n.$$ 
Define $T$ to be the subgroup $\langle (\sigma_1\cdots \sigma_{n-1})^n \rangle$ of $B_n$. It is clear from the geometric description of these generators that $T$ is contained in the center of $B_n$ (in fact, it is equal to the center) \cite[pp.246-47]{farbmargalit}. 

It follows from the Artin presentation of $B_n$ that the abelianization of $B_n$ is infinite cyclic and is generated by the class of any one of the generators $\sigma_j$. The abelianization map $B_n\rightarrow \ZZ$ sends the word $\sigma_{i_1}^{n_1}\cdots \sigma_{i_k}^{n_k}$  to $n_1+\cdots +n_k$.

\vspace{.1in} 
The following result will play a key role in our computations.
\begin{theorem}[Arnol'd \cite{arnold}]\label{braidcohomology}
For all $n\geq 2$, we have 
$$
H^j(B_n,\QQ) \cong \left\{
\begin{array}{lr}
\QQ & j = 0,1 \\
0 & j \geq 2
\end{array}
\right.
$$
\end{theorem}

\subsection{A long exact sequence} 
We are now ready to compute. 
\begin{lemma}\label{Mnlemma}
For all $n\geq 3$ and all $j\geq 1$, we have $H^j(M_n,\QQ) = 0$.
\end{lemma}
\begin{proof}
The Hochschild-Serre spectral sequence of the extension
$$1\rightarrow T  \rightarrow B_n\rightarrow M_n\rightarrow 1$$
degenerates to a long exact sequence
$$\cdots \rightarrow H^{j-1}(B_n,\QQ)\rightarrow H^{j-2}(M_n, H^1( T , \QQ))\rightarrow H^j(M_n,\QQ)\rightarrow H^j(B_n,\QQ)\rightarrow \cdots$$
Applying Theorem \ref{braidcohomology}, we see that we have isomorphisms
$$H^{j-1}(B_n,\QQ)\rightarrow H^{j-2}(M_n, H^1(T , \QQ)) \cong H^{j-2}(M_n,\QQ)$$
for all $j\geq 3$. By Theorem \ref{braidcohomology}, $H^{j}(B_n,\QQ)\cong 0$ for all $j\geq 2$. 
The long exact sequence also has a segment
\begin{equation}\label{braidsequence}
0\rightarrow H^1(M_n,\QQ)\rightarrow H^{1}(B_n,\QQ)\rightarrow H^{0}(M_n, H^1(T, \QQ))
\end{equation}
which we now analyze. The map 
\begin{equation}\label{braidequation}
H^{1}(B_n,\QQ)\rightarrow H^{0}(M_n, H^1(T, \QQ)) = H^1(T,\QQ)
\end{equation} 
is the dual of the map $H_1(T,\QQ)\rightarrow H_1(B_n,\QQ)$ induced by inclusion. The classes of the Artin generators $\sigma_j$ for $B_n$ all coincide in $H_1(B_n, \QQ)$, so this implies that the generator of $T$ of $H_1( T,\QQ)$ maps to $n(n-1)$ times a generator of $H_1(B_n,\QQ)$. Thus the map (\ref{braidequation}) is an isomorphism. By exactness of the sequence (\ref{braidsequence}), we have $H^1(M_n,\QQ) = 0$. 
\end{proof}
%
%
%
%
%
%
%
%
\begin{lemma}\label{modspherelemma}
Let $n\geq 0$ and suppose that $G$ is a subgroup of $\text{Mod}(S_{0,n})$ containing $M_{n-1}$. Then we have $H^j(G,\QQ) = 0$ for all $j\geq 1$.
\end{lemma}
\begin{proof}
The hypotheses imply that $M_n$ is a finite-index subgroup of $G$. Thus the natural maps
$$H^j(G,\QQ)\rightarrow H^j(M_{n-1},\QQ)$$
are injective. Lemma \ref{Mnlemma} then implies that $H^j(G,\QQ)\cong 0$  for all $j\geq 1$ provided  $n\geq 4$. It is known that $\text{Mod}(S_{0,n})$ is trivial when $n=0,1$, is isomorphic to $\ZZ/2\ZZ$ when $n=2$, and is isomorphic to the symmetric group $S_3$ when $n=3$ \cite[p.50]{farbmargalit}. Thus $G$ is finite in these cases. This immediately implies that $H^j(G,\QQ) = 0$ when $0\leq n\leq 3$.
\end{proof}

We are ready to prove the main result of this section.
\begin{proof}[Proof of Proposition \ref{section4mainprop}]
Let $H$ denote the deck group of $p$. Since $H$ is finite, the Hochschild-Serre spectral sequence of the extension 
$$1\rightarrow H\rightarrow \text{SMod}_p(S_g)\rightarrow \text{LMod}_p(S_{0,n})\rightarrow 1.$$
shows that
$$H^j(\text{SMod}_p(S_g),\QQ) \cong H^j(\text{LMod}_p(S_{0,n}),\QQ)$$
for all $j\geq 0$. Since we assumed $M_{n-1}\subset \text{LMod}_p(S_{0,n})$, Lemma \ref{modspherelemma} implies at once that $H^j(\text{SMod}_p(S_g),\QQ)\cong 0$ for all $j\geq 1$.
\end{proof}
\subsection{Cohomology with non-trivial coefficients}
In general, rather little seems to be known about the cohomology of the symmetric mapping class groups, especially with non-trivial coefficients. It appears to be an interesting and difficult problem to compute the cohomology of $\text{Mod}_H(S_g)$ with coefficients coming from an arbitrary (rational) representation of $\Sp_g(\QQ)$.

As a first step in this direction, we compute the cohomology of $\text{Mod}_H(S_g)$ with coefficients in $H^1(S_g,\QQ)$ under the assumption that $S_g/H \cong \PP^1$. We will not need to make use of these results in this paper, but we include them here in order to give a slightly more complete picture of the cohomology of $\text{Mod}_H(S_g)$. The uninterested reader can safely ignore this section.

\subsubsection{A vanishing result} Let $V$ be any $\QQ$-vector space with a $\text{Mod}_H(S_g)$-module structure. For simplicity, write $\mathcal{Q} = \text{Mod}_H(S_g)/H$. Then we have a Hochschild-Serre spectral sequence with $E_2$-page 
$$H^j(\mathcal{Q}, H^k(H,V)) \implies H^{j+k}(\text{Mod}_H(S_g),V).$$
Since $H$ is a finite group, we have $H^j(H,V) = 0$ for all $j>0$ by Maschke's Theorem. Thus the spectral sequence degenerates at the $E_2$-page giving isomorphisms 
$$H^j(\cQ, V^H)\cong H^j(\text{Mod}_H(S_g), V).$$
Thus if $V^H = 0$ the cohomology groups $H^j(\text{Mod}_H(S_g), V)$ all vanish. Although we do not know general conditions under which this holds, we completely understand what happens when $V = H^1(S_g,\QQ)$ and $S_g/H \cong \PP^1$. These assumptions imply that 
$$V^H  = H^1(S_g,\QQ)^H \cong H^1(\PP^1,\QQ) = 0.$$
We have therefore shown the following.
 \begin{proposition}\label{twistedcohomologyproposition}
Suppose that $g\geq 2$ and that $H < \text{Mod}(S_g)$ is a finite group with $S_g/H \cong \PP^1$. Then for all $j\geq 0$ we have 
$H^j(\text{Mod}_H(S_g), H^1(S_g,\QQ)) = 0$.
\end{proposition}

\subsubsection{Symplectic coefficients} For $g\geq 2$, let $\lambda_1,\ldots, \lambda_g$ denote a system of fundamental weights for the algebraic group $\Sp_g(\QQ)$. Let $n_1\geq n_2\geq \cdots \geq n_g\geq 0$ be integers and set $\lambda = \sum_{j=1}^gn_j\lambda_j$. Define $V(\lambda)$ to be the unique irreducible $\Sp_g(\QQ)$-module with highest weight $\lambda$. 

Every representation of $\Sp_g(\QQ)$ is also a representation of $\text{Mod}_H(S_g)$ by way of the composition 
$$\text{Mod}_H(S_g)\rightarrow \Sp_g(\ZZ)\hookrightarrow \Sp_g(\QQ).$$

There is a natural $\Sp_g(\ZZ)$-equivariant isomorphism $H_1(S_g,\QQ)\cong V(\lambda_1)$. Therefore by Proposition \ref{twistedcohomologyproposition}, we have $H^j(\text{Mod}_H(S_g), V(\lambda_1)) = 0$ whenever $S_g/H\cong \PP^1$. 

\vspace{.1in}
It would be interesting to compute $H^j(\text{Mod}_H(S_g), V(\lambda))$ for an arbitrary highest weight $\lambda$. It appears there are very few results in this direction, and most of what is known seems to concern the hyperelliptic mapping class groups. Since a hyperelliptic involution $\sigma$ acts on $V(\lambda_1)$ by $-\text{Id}$ the cohomology groups $H^j(\text{Mod}_{\langle \sigma\rangle}(S_g), V(\lambda)) = 0$ whenever $|\lambda|$ is odd. Besides this, few general results seem to exist beyond Tanaka's result \cite{tanaka} that $H_1(\text{Mod}_{\langle \sigma \rangle}(S_g), V(\lambda_1)^{\otimes 2}) = 0$ and Watanabe's result \cite{watanabe2016} that 
$$
H^1(\text{Mod}_{\langle \sigma\rangle}(S_2), V(\lambda)) = H^1(\text{Mod}(S_2), V(\lambda)) = 
\left\{
\begin{array}{lr}
\QQ& \lambda = 2\lambda_2 \\
0&\ \text{otherwise}
\end{array}
\right.
.$$

We remark here that, just as it can be shown that the (first) Johnson homomorphism $\tau_1: \cI(S_g)\rightarrow V(\lambda_3)$ gives rise to a non-trivial element of $H^1(\text{Mod}(S_g), V(\lambda_3))$ (see, for example, \cite{hain94}), it can be shown that the restriction of the second Johnson homomorphism (see \cite{hain2013}, \cite{morita})
$$\tau_2: \cK(S_g)\rightarrow V(2\lambda_2)$$ to the hyperelliptic Torelli group (see Section \ref{abelianizationGgsection}) gives rise to a non-trivial element of $H^1(\text{Mod}_{\langle \sigma \rangle}(S_g), V(2\lambda_2))$. Thus $H^1(\text{Mod}_{\langle \sigma \rangle}(S_g), V(2\lambda_2))\neq 0$ for $g\geq 2$.
\section{Covers of the projective line}\label{cycliccoversofp1section}
In this section, we prove Part 1 of Theorem \ref{maintheorem1}. As a corollary, we are able to compute the Picard groups of certain families of moduli spaces $\cM_g^{H}$ with the property that $S_g/H$ is homeomorphic to a sphere.

\subsection{The Chow Ring}

 We begin with some computations in the Chow ring of $\cM_g^H$. Recall that $\cM_g^H$ is the quotient of a smooth quasiprojective variety by a finite group.

\vspace{.1in}
Let $X$ be a complex algebraic variety. The Chow group $A_k(X)$ is the group of $k$-dimensional algebraic cycles on $X$ modulo rational equivalence. Let $A_*(X)_{\QQ} = A_k(X)\otimes \QQ$. It is shown in \cite{fulton} that if $G$ is a finite group acting on $X$, there is a natural isomorphism $A_k(X)_{\QQ}^{G}\cong A_k(G\backslash X)_{\QQ}$ for all $k\geq 0$.

We will make use of the following lemma. 

\begin{lemma}\label{abeliancoverlemma}
Let $p: S_g\rightarrow S_0$ be a regular finite-sheeted covering with abelian deck group $H$. Let $\textbf{P}\subset S_0$ denote the branch locus of $p$. Then we have $\text{PMod}(S_0, \textbf{P}) \vartriangleleft \text{LMod}_p(S_0, \textbf{{P}})$.
\end{lemma}
\begin{proof}
Recall that $\text{PMod}(S_0, \textbf{P})$ is generated by Dehn twists about simple closed curves surrounding pairs of points in $\textbf{P}$. Since these Dehn twists fix $\textbf{P}$ pointwise, Lemma 2.1 in \cite{ghaswalawinarski} directly implies that they all lift to $S_g$. We conclude that $\text{PMod}(S_0, \textbf{P})$ is a (normal) subgroup of $\text{LMod}_p(S_0, \textbf{P})$.
\end{proof}
The following result is then immediately deduced from Lemma \ref{abeliancoverlemma} and the Teichm\"{u}ller description of the moduli spaces.
\begin{proposition}
Let $p: S_g\rightarrow S_0$ be a finite-sheeted regular covering with abelian deck group $H$. Let $n$ denote the number of branch points of $p$. There is a finite Galois covering $\cM_{0,n}\rightarrow \cM_g^H$. As a consequence, we have $A_*(\cM_g^H)_{\QQ} \cong \QQ$, with a single generator in codimension 0.
\end{proposition}
\begin{proof}
Since $\cM_{0,n}$ is isomorphic to a complement of hyperplanes in $\CC^{n-3}$, we have $A_k(\cM_{0,n})$ vanishes except in degree $n-3$, where it is isomorphic to $\ZZ$. But $A_*(\cM_g^H)_{\QQ} \cong A_*(\cM_{0,n})^{G}_{\QQ}$ where $G = \text{LMod}_p(S_g)/\text{PMod}(S_0, \textbf{P})$. Thus $A_*(\cM_g^H)_{\QQ} \cong \QQ$.
\end{proof}


We are now ready to prove the first part of Theorem \ref{maintheorem1}.
\begin{proof}[Proof of Part \ref{part1} of Theorem \ref{maintheorem1}]
It follows from \cite{knopkraftvust} that $\text{Pic}\ \cM_g^H$ contains the Picard group of the variety $M_g^H$ underlying $\cM_g^H$ as a finite-index subgroup. Thus $\text{Pic}\ \cM_g^H$ is finite if and only if $\text{Pic}\ M_g^H$ is. But this follows at once from the fact that 
$$\text{Pic}\ M_g^H \otimes \QQ \cong A_{n-4}(M_g^H)_{\QQ} = 0.$$ If $H_1(\text{SMod}_p(S_g), \ZZ)$ is finite, then the first Chern class $\text{Pic}\ \cM_g^H\rightarrow H^2(\text{SMod}_p(S_g),\ZZ)$ is injective, and so maps isomorphically onto the torsion subgroup $H_1(\text{SMod}_p(S_g), \ZZ)$.
\end{proof}
\subsection{The lifting criterion of Ghaswala-Winarski}
As an application of Theorem \ref{maintheorem1}, we are able to compute $\text{Pic}\ \cM_g^H$ in certain cases when $S_g/H\cong \PP^1$. For technical reasons, we will have to impose certain restrictions on the branching behavior of the covering map $p: S_g\rightarrow S_g/H$. 

Let $p: S_g\rightarrow S_0$ be a cyclic covering of degree $d$ that is branched over a set of $n$ points. In \cite{ghaswalawinarski} Ghaswala-Winarski identified the precise conditions on $p$ under which the natural homomorphism
$$\text{SMod}_p(S_g)\rightarrow \text{Mod}(S_{0,n})$$
is surjective. We shall refer to such cyclic coverings as \emph{numerically admissible coverings}, as they are characterized by certain numerical restrictions on the monodromy of $p$. 

\vspace{.1in}
In \cite{ghaswalawinarski}, Ghaswala-Winarski showed how to construct algebraic models of numerically admissible coverings. Consider an irreducible plane curve of the form
\begin{equation}\label{cycliccurve}
y^d = (x-a_1)^{n_1}\cdots (x-a_k)^{n_k}.
\end{equation}
where $d\geq 2$, the $a_j$ are pairwise distinct, and each $n_j$ satisfies $1\leq n_j < d$. The curve (\ref{cycliccurve}) can be completed by adding finitely many points. The result is an irreducible complete algebraic curve that is, in general, singular. Normalizing (\ref{cycliccurve}) gives a smooth algebraic curve $C$. Furthermore, projecting (\ref{cycliccurve}) onto the $x$-axis induces a degree $d$ morphism $C\rightarrow \PP^1$ which is obviously a cyclic covering of degree $d$.

Ghaswala-Winarski provided the exact conditions on the tuple $(n_1,\ldots, n_k)$ under which this procedure gives a numerically admissible covering. 
\begin{theorem}[Ghaswala-Winarski \cite{ghaswalawinarski}]
The cyclic covering $C\rightarrow \PP^1$ constructed from (\ref{cycliccurve}) is numerically admissible if and only if one of the following is true:\begin{enumerate}
\item $n_1=\cdots = n_k$ and $k = 0\ \text{or}\ -1$\ \emph{mod} $d$
\item $d\geq 3$ and $k = 1$
\item $d\geq 3$, $k =2$ and $n_1 = -n_2$\ \emph{mod} $d$
\end{enumerate}
\end{theorem}

Since any hyperelliptic curve $C$ stems from a plane curve of the form 
$$y^2 = (x_1-a_1)\cdots (x_{2g+2}-a_{2g+2})$$
the hyperelliptic double cover $C\rightarrow \PP^1$ provides an example of a numerically admissible covering. 

\subsection{The abelianization of the symmetric mapping class group}

\subsubsection{Numerically admissible coverings}
The following result allows us to easily calculate the abelianization of $\text{Mod}_H(S_g)$ when $S_g\rightarrow S_g/H\cong \PP^1$ is a numerically admissible covering.

\begin{theorem}[Birman--Hilden \cite{birmanhilden}]\label{birmanhildentheorem}
Suppose that $H < \text{Mod}(S_g)$ is cyclic of order $d\geq 2$ and that the covering $S_g\rightarrow S_g/H$ is numerically admissible. The symmetric mapping class group $\text{Mod}_H(S_g)$ has a presentation with generators $t_1,\ldots, t_{n-1}$ and relations
\begin{equation*}
\left\{
  \begin{array}{lr}
    t_it_j = t_jt_i & |i-j| > 1\\
    t_it_{i+1}t_i = t_{i+1}t_it_{i+1} & 1\leq i \leq n-1\\
    (t_1\cdots t_{n-1}t_{n-1}\cdots t_1)^d = 1 &\ \\
   (t_1\cdots t_{n-1})^n = 1 &\ \\
    (t_1\cdots t_{n-1}, t_1) = 1 & \\
  \end{array}
\right.
\end{equation*}
\end{theorem}
\begin{corollary}\label{abelianizationcorollary}
Assume that $p: S_g\rightarrow S_0$ is a numerically admissible cyclic covering of degree $d\geq 2$ with deck group $H$. Then the abelianization of $\text{Mod}_H(S_g)$ is cyclic of order $(n-1)\text{gcd}(n,2d)$.
\end{corollary}
\begin{proof}
The first and second relations say that $H_1(\text{Mod}_H(S_g),\ZZ)$ generated by a single element, say the class $\overline{t}_1$ of $t_1$. The third relation says that $2d(n-1)\overline{t}_1 = 0$ and the fourth relation says that $n(n-1)\overline{t}_1 = 0$. From this, it follows that 
$H_1(\text{Mod}_H(S_g),\ZZ)$ is isomorphic to  $\ZZ$ modulo the subgroup generated by $2d(n-1)$ and $n(n-1)$. Thus  $H_1(\text{Mod}_H(S_g),\ZZ)$ is isomorphic to $\ZZ/(n-1)\text{gcd}(n,2d)\ZZ$.
\end{proof}

%
%
\subsubsection{Balanced Superelliptic Covers}
We now consider a class of cyclic branched covers $p:S_g\rightarrow S_0$ that have the Birman--Hilden property, but for which the canonical homomorphism $\text{SMod}_p(S_g)\rightarrow \text{Mod}(S_{0,n})$ is not surjective. 

In \cite{ghaswalawinarski2} Ghaswala--Winarski studied a class of cyclic covers which they call \emph{balanced superelliptic covers}; these are the cyclic covers that are branched over $2n+2$ points, where $n = g/(d-1)$.

By finding an explicit presentation for $\text{LMod}_p(S_g,\textbf{P})$, Ghaswala-Winarski were able to prove the following result. 

\vspace{.1in}

\begin{theorem}[Ghaswala--Winarski \cite{ghaswalawinarski2}]\label{balancedsuperelliptic}
Suppose that $p: S_g\rightarrow S_0$ is a balanced superelliptic cover of degree $d\geq 3$ branched over a set $\textbf{P}$ of $2n+2$ points. 

\begin{equation}
H_1(\emph{LMod}_p(S_0, \textbf{P}),\ZZ) \cong \left\{
\begin{array}{lr}
\ZZ/2\ZZ \oplus \ZZ/2\ZZ \oplus \ZZ/(n(n-1)^2)\ZZ& \text{$n$ odd} \\
 \ZZ/2\ZZ \oplus \ZZ/(2n(n-1)^2)\ZZ& \text{$n$ even} 
\end{array}
\right.
\end{equation}
\end{theorem}

\vspace{.1in}

\begin{remark}Since $p$ has the Birman--Hilden property, there is a short exact sequence
$$1\rightarrow H\rightarrow \text{SMod}_p(S_g)\rightarrow \text{LMod}_p(S_0,\textbf{P})\rightarrow 1.$$
Since $H$ is finite, Theorem \ref{balancedsuperelliptic} implies that $H_1(\text{SMod}_p(S_g),\ZZ)$ is finite. 
\end{remark}

\subsection{Picard groups of moduli spaces of cyclic covers of $\PP^1$}
By combining Corollary \ref{abelianizationcorollary} and Theorem \ref{balancedsuperelliptic} with Theorem \ref{maintheorem1}, we immediately obtain the following. 

\begin{corollary}
Let $g\geq 2$ and suppose that $H < \text{Mod}(S_g)$ is a finite cyclic group of order $d \geq 2$. Let $p$ denote the regular branched covering $S_g\rightarrow S_g/H\cong \PP^1$. 

\begin{enumerate}
\item If $p$ is a numerically admissible cover branched over $n$ points, then $\emph{Pic}\ \cM_g^H$ is cyclic of order $(n-1)gcd(2d,n)$.

\item If $p$ is a balanced superelliptic cover of degree $d\geq 3$ branched over $2n+2$ points with $n = g/(d-1)$, then 
\begin{equation*}
\emph{coker}\left(H \rightarrow \emph{Pic}\ \cM_g^H\right) \cong \left\{
\begin{array}{lr}
\ZZ/2\ZZ \oplus \ZZ/2\ZZ \oplus \ZZ/(n(n-1)^2)\ZZ& \text{$n$ odd} \\
 \ZZ/2\ZZ \oplus \ZZ/(2n(n-1)^2)\ZZ& \text{$n$ even} 
\end{array}
\right.
\end{equation*}
\end{enumerate}
\end{corollary}

\vspace{.1in}

\section{The higher genus case}\label{sectionhighergenus}
In this section, we focus on regular branched coverings $p:S_g\rightarrow S_h$ such that $h\geq 3$ as well as certain unbranched nilpotent coverings.

\subsection{Abelian symmetries}
We begin with the following result of Putman.
\begin{theorem}[Putman \cite{putmannote}]\label{putmanstheorem}
Assume that $g\geq 3$. Suppose that $\Gamma$ is a finite-index subgroup of $\text{PMod}(S_{g,n})$ that contains $K_{g,n}$. Then $H_1(\Gamma,\QQ) = 0$.
\end{theorem}
Let $(X, \textbf{P})$ denote a closed surface with a finite set $\textbf{P}$ of marked points. A separating simple closed curve on $(X,\textbf{P})$ is a simple closed curve on $X$ that is disjoint from $\textbf{P}$ and divides $X$ into two subsurfaces with exactly one boundary component each. We refer to a Dehn twist along such a curve as a separating twist. 

\begin{lemma}\label{dehntwistslift}
Suppose that $p:S_g\rightarrow S_h$ is a finite-sheeted regular covering with a set $\textbf{P}\subset S_h$ of $n$ branch points and abelian deck group $H$. Then any separating twist on $(S_h,\bf{P})$ lifts to a homeomorphism of $S_g$.
\end{lemma}

\begin{proof}
Define the punctured surfaces $S_h' = S_h\setminus \textbf{P}$, and let $S_g' = S_g\setminus p^{-1}(\textbf{P})$. Then $p$ restricts to a regular unbranched covering $p': S_g'\rightarrow S_h'$ with deck group $H$. Since $H$ is abelian, covering space theory shows that a homeomorphism of $S'_h$ lifts to $S_g'$ if and only if it preserves the kernel of the homomorphism $H_1(S_h',\ZZ)\rightarrow H$ determined by $p'$. 

Let $c$ be a separating simple closed curve on $(S_h,\textbf{P})$. Let $T_c$ denote a Dehn twist on $c$, and let $T_c'$ denote its restriction to $S_h'$. By classification of surfaces, it is possible to find a basis for $H_1(S_h',\ZZ)$ consisting of cycles that are disjoint from $c$. It follows that $T_c'$ on $c$ acts trivially on $H_1(S_h',\ZZ)$, and therefore lifts to a homeomorphism $\varphi'$ of $S_g'$. The canonical extension of $\varphi'$ to a homeomorphism of $S_g$ is then a lift of $T_c$. 
\end{proof}

\begin{corollary}\label{johnsoncorollary}
Under the hypotheses of Lemma \ref{dehntwistslift}, the liftable mapping class group $\emph{LMod}_p(S_h, \bf{P})$ contains the Johnson subgroup $\cK(S_{g,n})$.
\end{corollary}

\vspace{.1in}
\begin{proposition}\label{highergenusrationalabelianization}
Let $p: S_g\rightarrow S_h$ be a regular finite-sheeted abelian covering with $h\geq 3$. Then $H_1(\text{SMod}_p(S_g),\QQ) = 0$.
\end{proposition}
\begin{proof}
The liftable mapping class group $\text{LMod}_p(S_h,\textbf{P})$ has finite index in $\text{Mod}(S_h, \textbf{P})$. By Corollary \ref{johnsoncorollary}, it contains $\cK(S_{g,n})$. It follows that 
$$G_{h,n}: = \text{PMod}(S_h, \textbf{P})\cap \text{LMod}_p(S_h,\textbf{P})$$ 
is a finite-index subgroup of $\text{PMod}(S_h,\textbf{P})$ containing the Johnson subgroup $\cK(S_{h,n})$. By Theorem \ref{putmanstheorem}, $H_1(G_{h,n},\QQ) = 0$. On the other hand $G_{g,n}$ is a finite-index subgroup of $\text{LMod}_p(S_h,\textbf{P})$, so the natural map
$$ 0 \cong H_1(G_{h,n},\QQ)\rightarrow H_1(\text{LMod}_p(S_h,\textbf{P}),\QQ)$$
is a surjection. Finally, since $H_0$ is finite, the 5-term exact sequence associated to the extension 
$$1\rightarrow H\rightarrow \text{SMod}_p(S_g)\rightarrow \text{LMod}_p(S_h,\textbf{P})\rightarrow 1$$
immediately gives that $H_1(\text{SMod}_p(S_g),\QQ) = 0$.
\end{proof}

\vspace{.1in}

\begin{corollary}\label{finiteabelianization}
Under the hypotheses of Proposition \ref{highergenusrationalabelianization}, $H_1(\emph{SMod}_p(S_g),\ZZ)$ is finite. 
\end{corollary}
It is now a simple matter to prove Part \ref{part2} of Theorem \ref{maintheorem1}.
\begin{proof}[Proof of Theorem \ref{maintheorem1}]
Since $\mathfrak{X}_g$ is contractible, Proposition \ref{equivariantcohomologyprop} implies that 
$$H^2_{\text{Mod}_H(S_g)}(\mathfrak{X}^H_g,\ZZ) \cong H^2(\text{Mod}_H(S_g),\ZZ).$$ 
 Together, Theorem \ref{hainpic0theorem} and Corollary \ref{finiteabelianization} imply that the first Chern class 
$$\text{Pic}\ \cM_g^{H}\rightarrow H^2(\text{SMod}_p(S_g),\ZZ)$$
is injective. Since $\text{SMod}_p(S_g)$ is finitely presented,  $H^2(\text{SMod}_p(S_g),\ZZ)$ is finitely generated; thus $\text{Pic}\ \cM_g^{H}$ is finitely generated. Theorem \ref{putmanpicardtheorem} implies that the image of $c_1$ contains the torsion subgroup of $H^2(\text{SMod}_p(S_g),\ZZ)$, which by Corollary \ref{finiteabelianization} is isomorphic to $H_1(\text{SMod}_p(S_g),\ZZ)$.
\end{proof}

\subsection{Nilpotent symmetries}\label{nilpotentsymmetries}
So far, we have only discussed the symmetric mapping class groups $\text{Mod}_H(S_g)$ with $H$ abelian. However, there is nothing in our general setup that prevents us from studying the behavior of these groups, or the corresponding moduli spaces of symmetric curves, when $H$ is nonabelian.

In this section, we will show that, under certain numerical restrictions, $\cM_g^H$ has finitely generated Picard group provided $H$ is nilpotent and the action of $H$ on $S_g$ is free. Our proof makes use of recent work of Ershov--He \cite{ershovhe}.

Beyond the cases where $H$ is nilpotent, we do not know of any other types of nonabelian groups for which the Picard group of $\cM_g^H$ is finitely generated. The following is a natural question. 

\begin{question}\label{question1}
For which nonabelian groups $H < \text{Mod}(S_g)$ do the moduli spaces $\cM_g^H$ have finitely generated Picard groups?
\end{question}
By Theorem \ref{hainpic0theorem}, finite generation of $\cM_g^H$ would follow from the vanishing of $H_1(\text{Mod}_H(S_g),\QQ)$. Assume that $H$ acts freely on $S_g$ and that the quotient surface $S_h = S_g/H$ has genus at least 3. Since $\text{Mod}_H(S_g)/H$ is isomorphic to a finite-index subgroup of $\text{Mod}(S_h)$ the following conjecture of Ivanov is quite relevant.

\begin{conjecture}[Ivanov \cite{ivanov}]\label{ivanovconjecture}
For $g\geq 3$ and any finite-index subgroup $\Gamma < \text{Mod}(S_g)$, we have $H_1(\Gamma,\QQ) = 0$.
\end{conjecture}

Hain subsequently showed \cite{hain94} that Ivanov's conjecture is true if $\cI(S_g) < \Gamma$. Putman \cite{putmannote} later strengthened this to include $\Gamma$ such that $\cK(S_g) < \Gamma$. A recent result of Ershov--He  \cite{ershovhe} implies the conjecture for subgroups containing a term of the lower central series of $\cI(S_g)$ in an explicit range of genera. 

Of course, the truth of Ivanov's full conjecture would imply that the Picard group of $\cM_g^H$ is finitely generated for all $g\geq 3$ and all finite $H < \text{Mod}(S_g)$.

\subsubsection{The Johnson filtration}\label{johnsonflltration} Fix a basepoint $x_0$ on $S_g$, and set $\pi = \pi_1(S_g,x_0)$. For each integer $k\geq 1$, let $\pi^{(k)}$ denote the $k$th term of the lower central series of $\pi$. 
The action of homeomorphisms on $S_g$ induces a homomorphism 
$$\text{Mod}(S_g)\rightarrow \text{Out}\left(\pi/\pi^{(k+1)}\right)$$
for each $k\geq 1$. We define $\cI(S_g)(k)$ to be the kernel of this homomorphism. This produces a filtration $\{\cI(S_g)(k)\}_{k\geq 0}$ of $\text{Mod}(S_g)$ known as the \emph{Johnson filtration}. We recover the Torelli group $\cI(S_g)$ by setting $k=1$. By a theorem of Johnson \cite{johnson2}, we have $\cK(S_g) = \cI(S_g)(2)$.

\subsubsection{The liftable mapping class group of a nilpotent cover}
 We now suppose that $g\geq 3$ and that $p: S_h\rightarrow S_g$ is an \emph{unbranched} covering with nilpotent deck group $H$. Then $p$ is determined by a surjective homomorphism $\varphi: \pi \rightarrow H$. If $H$ has nilpotency class $k$, i.e. $k$ is the smallest positive integer such that $H^{(k+1)} = 1$, then $\varphi$ factors through the nilpotent quotient $\pi/\pi^{(k+1)}$ to give a homomorphism $\overline{\varphi}: \pi/\pi^{(k+1)}\rightarrow H$.

\begin{proposition}\label{nilpotentcoveringlemma}
Suppose $p: S_h\rightarrow S_g$ is an unbranched nilpotent covering of nilpotency class $k$. Then $\emph{LMod}_p(S_g)$ contains $\cI(S_g)(k)$.
\end{proposition}
\begin{proof}
By covering space theory, if a homeomorphism of $(S_g,x_0)$ preserves the kernel of $\varphi$, it lifts to a homeomorphism of $S_h$. 

Let $f$ be a homeomorphism of $S_g$ with the property that $[f]\in \cI(S_g)(k)$. It follows from \cite[Corollary 2]{kordekfiniteness} that we may isotope $f$ to a homeomorphism $f'$ such that $f'(x_0) = x_0$ and such that the induced automorphism $f'_*$ of $\pi$ satisfies $f_*'(\gamma)\gamma^{-1}\in \pi^{(k+1)}$ for all $\gamma\in \pi$. 

For each $\gamma\in \pi$, there exists some $\gamma_k\in \pi^{(k+1)}$ such that $f'_*(\gamma) = \gamma_k\cdot \gamma$. Suppose that $\varphi(\gamma) = 1$. Then we have 
\begin{align*}
\varphi \left( f'_*(\gamma)\right) = \varphi \left(\gamma_k\cdot \gamma\right) &= \varphi \left(\gamma_k \right)\cdot \varphi \left(\gamma\right)\\
&= \varphi \left(\gamma_k \right)\\
&=1
\end{align*}
since the terms of the lower central series are characteristic and $H^{(k+1)} = 1$ by assumption. Thus $f'$ lifts to a homeomorphism of $S_h$. Thus $[f'] = [f]\in \text{LMod}_p(S_g)$.
\end{proof}

\begin{corollary}
Suppose that $H < \text{Mod}(S_h)$ is a finite nilpotent group of nilpotency class $k\geq 2$ and set $S_h/H = S_g$. If $g\geq 8k-4$ and the covering $p:S_h\rightarrow S_g$ is unbranched, then we have $H_1(\text{Mod}_H(S_h),\QQ) = 0$. In particular, $\text{Pic}\ \cM_h^H$ is finitely generated.
\end{corollary}
\begin{proof}
It is well known that the $k$th term $\cI^{(k)}(S_g)$ of the lower central series of $\cI(S_g)$ is contained in $\cI(S_g)(k)$. Thus the liftable mapping class group of $p$ contains $\cI^{(k)}(S_g)$. Theorem 1.9 in \cite{ershovhe}, although stated for surfaces with boundary, implies quite directly, then, that the abelianization of $\text{LMod}_p(S_g)$ is finite provided $g\geq 8k-4$. It immediately follows that $H_1(\text{Mod}_H(S_h),\QQ) = 0$ and that $\text{Pic}\ \cM_g^H$ is finitely generated.
\end{proof}

\section{Curves of compact type}\label{curvesofcompacttypesection} 
The Deligne-Mumford compactification $\overline{\cM}_g$ of $\cM_g$ is a projective variety with finite quotient singularities. Its points parametrize isomorphism classes of stable genus $g$ curves with at worst nodal singularities. The boundary $\partial \cM_g =\overline{\cM}_g\setminus \cM_g$ is a divisor. There is an irreducible decomposition
$$\partial \cM_g = \Delta_{irr} \cup \bigcup_{j=1}^{\floor{g/2}}D_j.$$
The generic point of $\Delta_{irr}$ represents an irreducible curve with a single node, and the generic point of $D_j$ represents a reducible curve with two smooth components of genus $j$ and $g-j$ and a single node. 

\vspace{.1in}
Of principal interest to us are the \emph{curves of compact type}. This is a stable nodal curve $C$ all of whose irreducible components are smooth and whose dual graph is a tree. The generalized Jacobian of $C$ is a compact complex torus and is isomorphic to the product of the Jacobians of the smooth components. The moduli space of curves of compact type is the open subvariety 
$$\cM_g^c = \overline{\cM}_g\setminus \Delta_{irr} \subset \overline{\cM}_g.$$

\subsection{Level structures}
Assume that $g\geq 2$, and let $C$ be a genus $g$ curve of compact type. Assume that $C = C_1\cup \cdots \cup C_n$, where each $C_j$ is a smooth curve. Let $\omega_j$ denote the symplectic form on $H_1(C_j,\ZZ)$. Since $H_1(C,\ZZ) = \displaystyle \bigoplus H_1(C_j,\ZZ)$, we obtain a symplectic form $\omega$ on $H_1(C,\ZZ)$ by setting $\omega = \omega_1\oplus \cdots \oplus\omega_n$. This induces a symplectic form on $H_1(C,\ZZ/m\ZZ)$ for each $m$.

Now fix an integer $m\geq 0$. A \emph{mod $m$ homology framing} for $C$ is a symplectic basis $\cF = \{a_1,\ldots, a_g, b_1,\ldots b_g\}$ for $H_1(C,\ZZ/m\ZZ) \cong (\ZZ/m\ZZ)^{2g}$. A level $m$ structure on a curve of compact type $C$ is a choice of mod $m$ homology framing. Two curves $C_1,C_2$ of compact type with level $m$ structures $\cF_1, \cF_2$ are said to be isomorphic if there is an isomorphism $C_1\xrightarrow{\cong} C_2$ that carries $\cF_1$ to $\cF_2$.

\subsubsection{Torelli spaces}
Torelli space $\cT_g^c$ is the moduli space of curves of compact type equipped with a level $0$ structure, i.e. an integral homology framing. It is a complex manifold, but it is not biholomorphic to an algebraic variety. The open submanifold $\cT_g$ parametrizing smooth curves is the complement of a divisor with normal crossings; it can also be realized as the quotient $\cI(S_{g})\backslash \mathfrak{X}_g$.

The symplectic group $\Sp_g(\ZZ)$ acts properly and virtually freely on $\cT_{g}^c$ and $\cT_{g}$ via its action on framings. There is a natural holomorphic map $\cT_{g}^c\rightarrow \cM_{g}^c$ obtained by forgetting the framings. By \cite{hain2006} this induces natural identifications $\Sp_g(\ZZ)\backslash \cT_{g}^c\cong \cM_{g}^c$ and $\Sp_g(\ZZ)\backslash \cT_{g} \cong \cM_{g}$ which we use to endow both quotients with an algebraic structure.

\subsubsection{Finite level covers of $\cM_g^c$} 
When $m\geq 3$, the level $m$ subgroup $\Sp_g(\ZZ)[m]$ of $\Sp_g(\ZZ)$ is torsion-free. The quotient space $\Sp_g(\ZZ)[m]\backslash \cT_{g}^c$ is therefore smooth, and in fact furnishes a quasiprojective finite cover of $\cM_{g}^c$. 
For each $m\geq 1$, the moduli space $\cM_{g}^c[m]$ of genus $g$ curves of compact type with level $m$ structure can be identified with $\Sp_g(\ZZ)[m]\backslash \cT_{g}^c$; we shall regard it as a quasiprojective orbifold.  

Similarly, the moduli space $\cM_{g}[m]$ of smooth genus $g$ curves with level $m$ structure can be realized as the quasiprojective orbifold\footnote{Here we temporarily relax our requirement that a quasiprojective orbifold $(X,G)$ satisfies $\pi_1(X,*) = 1$.} 
$(\Sp_g(\ZZ)[m], \cT_{g})$; as noted above, it is also realizable as the quasiprojective orbifold $(\text{PMod}(S_{g})[m], \mathfrak{X}_{g})$. 

\vspace{.1in} From this descriptions given above, we see that the natural maps $\cM_{g}[m]\rightarrow \cM_{g}$ and $\cM_{g}^c[m]\rightarrow \cM_{g}^c$ obtained by forgetting the framings are finite Galois covers with Galois group $\Sp_g(\ZZ/m\ZZ)$.

\subsection{The Picard group of $\cM_{g}^c[m]$}
\vspace{.1in}
Using Johnson's work on the Torelli group, Hain \cite{hain94} showed that the Picard groups of the $\cM_{g}[m]$ are finitely generated, provided $g\geq 3$. Building on this work, Putman \cite{putmanpicard} computed the Picard groups of the $\cM_g[m]$ exactly, subject to some restrictions on the level $m$. 

\vspace{.1in}
Hain's result readily implies that $\text{Pic}\ \cM_g^c[m]$ is finitely generated. In this section, we compute the torsion subgroup of $\text{Pic}\ \cM_g^c[m]$ for $m\geq 3$. Since these results will not be used elsewhere in the paper, the uninterested reader may safely skip this section.

\vspace{.1in}
We will need to make use of the following result in our computation.
%
%
%

\begin{proposition}[Putman \cite{putmanpicard}, Sato \cite{sato}]
For each $g\geq 2$ and $m\geq 3$ we have
$$H_1(\Sp_g(\ZZ)[m],\ZZ)\cong
\left\{
\begin{array}{lr}
(\ZZ/m\ZZ)^{\oplus g(2g+1)}&\ m\ \text{odd} \\
(\ZZ/m\ZZ)^{\oplus g(2g-1)}\oplus (\ZZ/2m\ZZ)^{\oplus 2g}&\ m\ \text{even}
\end{array}
\right.
$$
\end{proposition}
The following is an adaptation of Hain's result to the moduli spaces $\cM_{g}^c[m]$ of curves of compact type. 
\begin{proposition}
For all $g\geq 3$ and $m\geq 1$, $\text{Pic}\ \cM_{g}^c[m]$ is finitely generated.
\end{proposition}
\begin{proof}
Assume first that $m\geq 3$. The required result can be proved in many ways, all of which invoke the fact that the complement of $\cM_{g}[m]$ in $\cM_{g}^c[m]$ is a divisor $D$. The most direct way is to first observe that there is an exact sequence
$$\bigoplus \ZZ \rightarrow \text{Pic}\ \cM_{g}^c[m]\rightarrow  \text{Pic}\ \cM_{g}[m]\rightarrow 0$$
where the direct sum runs over the set of components of $D$. Since this exhibits $\text{Pic}\ \cM_{g}^c[m]$ as an extension of a finitely generated group by a finitely generated group, it is also finitely generated.

To handle the cases $m=1,2$, observe that, in either case, we can choose $m'\geq 3$ such that $\Sp_g(\ZZ)[m'] <\Sp_g(\ZZ)[m]$. Then $\cM_{g}^c[m']$  is a regular quasiprojective finite cover of $\cM_{g}^c[m]$. We now set
$\Gamma = \Sp_g(\ZZ)[m]/\Sp_g(\ZZ)[m']$. The results of the preceding paragraph then imply that 
$$H^1_{\Gamma}(\cM_{g}^c[m'],\QQ) = H^1(\cM_{g}^c[m'],\QQ)^{\Gamma} = 0.$$
Thus $\text{Pic}\ \cM_{g}^c[m]$ is finitely generated by Theorem \ref{hainpic0theorem}.
\end{proof}


Since $\cT_{g}$ is the complement in $\cT_{g}^c$ of a normal crossings divisor, a monodromy argument shows that $\cT_{g}^c$ has fundamental group $\cI(S_g)/\cK(S_g)$ (see, for example, Hain \cite{hain2006}). Assume now that $m\geq 3$. At the level of analytic spaces, $\cM_{g}^c[m] = \Sp_g(\ZZ)[m]\backslash \cT_{g}^c$. Covering space theory shows that there is an extension
$$1\rightarrow \cI(S_g)/\cK(S_g)\rightarrow \pi_1(\cM_{g}^c[m],*)\rightarrow \Sp_g(\ZZ)[m]\rightarrow 1.$$
The action of $\Sp_g(\ZZ)[m]$ on $\cI(S_g)/\cK(S_g)$ is induced by the usual action of $\Sp_g(\ZZ)$ on $H_1(\cI(S_g),\ZZ)$. The 5-term exact sequence of rational homology for this extension has a segment
\begin{equation}\label{boundaryexactsequence}
\left(\cI(S_g)/\cK(S_g)\right)_{\Sp_g(\ZZ)[m]}\rightarrow H_1(\pi_1(\cM_{g}^c[m],*), \ZZ)\rightarrow H_1(\Sp_g(\ZZ)[m],\ZZ)\rightarrow 0.
\end{equation}
Work of Johnson \cite{johnson2} implies that there is a $\text{Mod}(S_g)$-equivariant isomorphism
$$\cI(S_g)/\cK(S_g)\rightarrow \Lambda^3V/\theta\wedge V$$
where $V = H_1(S_g,\ZZ)$ and $\theta\in \Lambda^2V$ is the standard symplectic form.
For each integer $m\geq 1$, define $V_m = H_1(S_g,\ZZ/m\ZZ)$. The following result allows us to compute the leftmost term of (\ref{boundaryexactsequence}) in terms of $V_m$. 

\vspace{.1in}
In what follows, we let $\Lambda^3_0V_m =\Lambda^3V_m/\theta\wedge V_m$.
\begin{proposition}[Putman \cite{putmanpicard}]
For all $m\geq 1$ and $g\geq 3$, we have isomorphisms 
$$(\Lambda^3_0V)_{\Sp_g(\ZZ)[m]}\cong \Lambda^3_0V_m .$$
\end{proposition}

\subsubsection{The relative Johnson homomorphisms}
For $g\geq 3$, Broaddus--Farb--Putman \cite{broaddusfarbputman} and Sato \cite{sato} have constructed ``relative Johnson homomorphisms"
$$\tau_m: \text{Mod}(S_g)[m]\rightarrow \Lambda^3_0V_m. $$
These homomorphisms generalize the classical Johnson homomorphisms first considered by Johnson in \cite{johnson2}. In particular, $\tau_m$ vanishes on $\cK(S_g)$. 

\vspace{.1in}
Observe that for each $m\geq 3$ there is an exact sequence
\begin{equation}\label{H1sequence}
\Lambda^3_0V_m \rightarrow H_1(\cM_{g}^c[m],\ZZ)\rightarrow H_1(\Sp_g(\ZZ)[m],\ZZ)\rightarrow 0.
\end{equation}
Covering space theory shows that for $m\geq 3$ there are isomorphisms 
$$\pi_1(\cM_{g}^c[m],*)\cong \text{PMod}_{g}[m]/\cK(S_g).$$ By imitating the arguments in \cite{putmanpicard, sato}, it is readily demonstrated that the leftmost map in (\ref{H1sequence}) is injective and is split by the relative Johnson homomorphism. Thus we have proved the following. 
\begin{proposition}
For each $g\geq 3$, $m\geq 3$ there is an isomorphism
$$H_1(\cM_{g}^c[m],\ZZ) \cong \Lambda^3_0V_m\oplus H_1(\Sp_g(\ZZ)[m],\ZZ).$$
\end{proposition}
Since this shows that $H_1(\cM_{g}^c[m],\ZZ)$ is finite, we are able to immediately deduce the following result.  
\begin{corollary}
For each $g\geq 3$, $m\geq 3$, the torsion subgroup of $\emph{Pic}\ \cM_{g}^c[m]$ is isomorphic to $\Lambda^3_0V_m\oplus H_1(\emph{\text{Sp}}_g(\ZZ)[m],\ZZ)$.
\end{corollary}
\section{Hyperelliptic curves}\label{hyperellipticsection}
Let $\cM_g^{hyp}[m]$ denote the preimage of $\cHH_g$ in $\cM_g[m]$ under the Galois covering 
$$\cM_g[m]\rightarrow \cM_g.$$ When $m$ is odd $\cM_g^{hyp}[m]$ consists of a single component. When $m$ is even, the number of components of $\cM_g^{hyp}[m]$ is equal to
$$
\frac{2^{g^2}\prod_{k=1}^g 2^{2k} -1}{(2g+2)!}.
$$

The set of irreducible components is acted upon transitively by the Galois group $\Sp_g(\ZZ/m\ZZ)$. We now fix a component of $\cM_g^{hyp}[m]$ and denote it by $\cHH_g[m]$. 

\subsection{Hyperelliptic curves with level structures}
Since $\cHH_g$ is a normal variety and the canonical holomorphic map $\cHH_g[m]\rightarrow \cHH_g$ is proper with finite fibers, the Generalized Riemann Existence Theorem implies that $\cHH_g[m]$ admits a unique algebraic structure making this map a finite morphism of varieties. In fact, for all $m\geq1$, each component $\cHH_g[m]$ is an affine variety. To see this, note that $\cHH_g$ is affine because it is the quotient of the affine variety $\cM_{0,2g+2}$ by $S_{2g+2}$.  Since $\cHH_g[m]\rightarrow \cHH_g$ is a finite morphism, it is affine \cite{hartshorne}. 

We now describe $\cHH_g[m]$ analytically. Fix a hyperelliptic involution $\sigma$ on the reference surface $S_g$. Let $\Delta_g$ denote the hyperelliptic mapping class group, i.e. the symmetric mapping class group of the branched cover $S_g\rightarrow S_g/\langle \sigma \rangle$. For each positive integer $m$ we define the level $m$ hyperelliptic mapping class group by
$$\Delta_g[m]= \Delta_g\cap \text{Mod}(S_g)[m].$$
There is a biholomorphism $\Delta_g[m]\backslash\mathfrak{X}_g^{\langle \sigma \rangle}\rightarrow \cHH_g[m]$. Since $\mathfrak{X}_g^{\langle \sigma \rangle}$ is contractible, this shows that $\cHH_g[m]$ is an Eilenberg-MacLane space for $\Delta_g[m]$ whenever $\Delta_g[m]$ is torsion-free, i.e. when $m\geq 3$.

%
%
%

%
%
%

\subsection{Hyperelliptic curves of compact type}
In \cite{acampo}, A'Campo studied the monodromy  of families of smooth hyperelliptic curves in order to prove results about the structure of the group 
$$G_g= \text{Im}\left(\Delta_g\rightarrow \Sp_g(\ZZ)\right).$$
It is easy to see, for example from the Humphries generating set \cite{farbmargalit} for $\text{Mod}(S_2)$, that $\Delta_2 = \text{Mod}(S_2)$. Thus $G_2 = \Sp_2(\ZZ)$. However, when $g\geq 3$, $G_g$ is a proper subgroup of $\Sp_g(\ZZ)$. We have the following result.
\begin{theorem}[A'Campo \cite{acampo}]\label{acampostheorem} 
For all $g\geq 2$, $G_g$ contains the level subgroup $\Sp_g(\ZZ)[2]$ 
and $G_g/\Sp_g(\ZZ)[2]$ is isomorphic to the symmetric group on the set of $2g+2$  fixed points of $\sigma$. 
\end{theorem}
The hyperelliptic locus $\cT_g^{c,hyp}$ in $\cT_g^c$ consists of finitely many mutually isomorphic smooth components. The symplectic group $\Sp_g(\ZZ)$ acts transitively on the set of these components. The stabilizer in $\Sp_g(\ZZ)$ of any one of them is isomorphic to the image of the natural map $\Delta_g\rightarrow \Sp_g(\ZZ)$. For a suitable choice of component $\cHH_g^c[0]$, the hyperelliptic locus $\cHH_g^c$ in $\cM_g^c$ is biholomorphic to the quotient $G_g\backslash \cHH^c_g[0]$. Since $\cHH_g^c[0]$ is simply-connected by \cite{brendlemargalitputman}, we view $\cHH_g^c$ as the orbifold $(G_g,\cHH_g^c[0])$. In fact, $\cHH_g^c$ is a quasiprojective orbifold, as we now explain. 

\vspace{.1in}
Let $\cHH^c_g[m]$ denote the image of $\cHH_g^c[0]$ in $\cM_g^c[m]$. Then $\cHH_g^c[m]$ is a smooth subvariety provided $m\geq 3$. It is biholomorphic to the quotient $G_g[m]\backslash \cHH_g^c[0]$, where we define 
$$G_g[m]= G_g\cap \Sp_g(\ZZ)[m].$$ 
\noindent The orbifold $(G_g[m], \cHH_g^c[0])$ is then a regular quasiprojective finite cover of $\cHH_g^c$.

\vspace{.1in}
Since $\cHH_g^c[0]$ is simply-connected, we have, for $m\geq 3$, a natural isomorphism 
\begin{equation}\label{H1hyplocuslevel}
H_1(\cHH_g^c[m],\ZZ)\cong H_1(G_g[m],\ZZ).
\end{equation}

\subsection{The abelianization of $G_g$} \label{abelianizationGgsection} In order to compute $\text{Pic}\ \cHH_g^c$, we need information about the abelianization of $G_g$. Our strategy involves the use of the \emph{hyperelliptic Torelli group} $S\cI(S_g)$, which is defined to be the kernel of the homomorphism $\Delta_g\rightarrow \Sp_g(\ZZ)$.

\vspace{.1in}
In this section, use properties of $S\cI(S_g)$ in order to compute $H_1(G_g,\ZZ)$. The basic idea is to analyze the exact sequence 
$$H_1(S\cI(S_g), \ZZ)\rightarrow H_1(\Delta_g,\ZZ)\rightarrow H_1(G_g,\ZZ)\rightarrow 0$$
coming from the exact sequence
$$1\rightarrow S\cI(S_g)\rightarrow \Delta_g\rightarrow G_g\rightarrow 1.$$

We will need to make use of some facts about the \emph{braid Torelli group} $\cB\cI_{n}$, which is the kernel of the Burau representation of $B_n$ at $t=-1$ \cite{brendlemargalitputman}. The following result gives characterization of $\cB\cI_n$ that we will find particularly useful.
\begin{theorem}[Brendle-Margalit-Putman \cite{brendlemargalitputman}]
Every element of $B\cI_{2g+1}$ can be written as a product squares of Dehn twists on simple closed curves enclosing $3$ or $5$ marked points.
\end{theorem}



Let $S_g^1$ denote a surface with a single boundary component, which is preserved by the hyperelliptic involution $\sigma$. The hyperelliptic mapping class group $\Delta_g^1$ is the subgroup of $\text{Mod}(S_g^1)$ consisting of isotopy classes of homeomorphisms of $S_g^1$ that commute with $\sigma$. In the definition of $\text{Mod}(S_g^1)$, we require all homeomorphisms and isotopies to restrict to the identity on the boundary of $S_g^1$. Note that $\sigma$ does not give an element of $\text{Mod}(S_g^1)$.

The hyperelliptic Torelli group $S\cI(S^1_g)$ is the subgroup of $\Delta_g^1$ acting trivially on $H_1(S_g^1,\ZZ)$. The hyperelliptic involution $\sigma$ has $2g+1$ fixed points, and the quotient $S_g^1/\langle \sigma \rangle$ can be identified with the marked disk $D_{2g+1}$. 
\begin{figure}[h]
\centering
\hspace{.3in} \includegraphics[scale=.5, angle =-90]{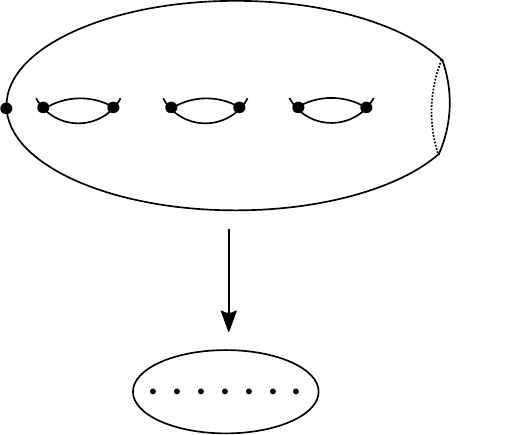}
\caption{The quotient map $S_g^1\rightarrow D_{2g+1}$ with branch points indicated.}
\end{figure}

The Birman--Hilden theory introduced in Section \ref{sectionsymmetricmappingclassgroups} can also be developed for surfaces with boundary. One consequence of this theory is an isomorphism $\Delta_g^1 \cong B_{2g+1}$ given by lifting homeomorphisms of $D_{2g+1}$ to (symmetric) homeomorphisms of $S_g^1$. 
In \cite{brendlemargalitputman}, it is shown that this isomorphism induces an isomorphism 
$$\cB\cI_{2g+1}\cong S\cI(S_g^1).$$
Capping the boundary of $S_g^1$ induces a surjection $S\cI(S_g^1)\rightarrow S\cI_g$ with cyclic kernel generated by the Dehn twist on the boundary. 

The relationships between all of these groups are summarized in the following diagram:
\begin{equation}\label{bigdiagram}
\xymatrix{
& B\cI_{2g+1}\ar[rr]\ar[ld]_{\cong}\ar[dd]& &B_{2g+1}\ar[dd]\ar[rd]^{\cong} & \\
S\cI(S_g^1)\ar[rd]_{\text{capping}} & && &\Delta_g^1\ar[ld]^{\text{capping}}\\
 &S\cI(S_g)\ar[rr] && \Delta_g&
}
\end{equation}

%
%
%
\begin{lemma}\label{hyperelliptictorellilemma}
The image of any element $x\in S\cI_g$ in $H_1(\Delta_g,\ZZ)$ has order dividing $2g+1$.
\end{lemma}
\begin{proof}
It is pointed out in \cite{brendlemargalitputman} that the image in the abelianization $H_1(B_{2g+1},\ZZ)$ of the square $T_c^2$ of a twist on a simple closed curve $c$ is equal to $12$ if $c$ encloses $3$ points and is equal to $40$ if $c$ encloses $5$ points. By commutativity of (\ref{bigdiagram})
this implies that the image in $H_1(\Delta_g,\ZZ)$ of every $x\in S\cI_g$ is a multiple of $4$. On the other hand, $H_1(\Delta_g,\ZZ)$ is cyclic of order $2(2g+1)$ if $g$ is even and $4(2g+1)$ if $g$ is odd. This implies that the image of every such $x$ is annihilated by $2g+1$. 
\end{proof}

\vspace{.1in}
\begin{proposition}\label{abelianizationGg}
For all $g\geq 2$, the abelianization of $G_g$ is cyclic of order $2$ if $g$ is even and order $4$ if $g$ is odd. 
\end{proposition}
\begin{proof}
The result of Theorem \ref{acampostheorem}, the 5-term exact sequence of the extension
$$1\rightarrow \Sp_g(\ZZ)[2] \rightarrow G_g\rightarrow S_{2g+2}\rightarrow 1,$$
the fact that $H_1(\Sp_g(\ZZ)[2],\ZZ)$ is 2-torsion \cite{brendlemargalitlevel4}, and the fact that $H_1(S_{2g+2},\ZZ)\cong \ZZ/2\ZZ$ together imply that $H_1(G_g,\ZZ)$ consists only of torsion elements whose order is a power of 2. The exact sequence
$$H_1(S\cI_g,\ZZ)\rightarrow H_1(\Delta_g,\ZZ)\rightarrow H_1(G_g,\ZZ)\rightarrow 0$$
along with Lemma \ref{hyperelliptictorellilemma} implies that $H_1(G_g,\ZZ)$ is a quotient of $\ZZ/2\ZZ$ when $g$ is even and a quotient of $\ZZ/4\ZZ$ when $g$ is odd. On the other hand, Theorem \ref{acampostheorem} implies that $G_g$ surjects onto $S_{2g+2}$. Since $H_1(S_{2g+2},\ZZ)\cong \ZZ/2\ZZ$, $H_1(G_g,\ZZ)$ surjects onto $\ZZ/2\ZZ$. This shows that, when $g$ is even, there there is an isomorphism $H_1(G_g,\ZZ)\cong \ZZ/2\ZZ$. 

When $g$ is odd, the presentation for $\Delta_g$ given in Theorem \ref{birmanhildentheorem} shows that the class $\overline{\sigma}\in H_1(\Delta_g,\ZZ)$ of the hyperelliptic involution $\sigma$ is equal to $2(2g+1)$ times a generator, whereas $H_1(\Delta_g,\ZZ)$ is cyclic of order $4(2g+1)$. Thus $\overline{\sigma}\neq 0$. Moreover, $\overline{\sigma}$ it is not in the image of the map $H_1(S\cI_g,\ZZ)\rightarrow H_1(\Delta_g,\ZZ)$ because $(2g+1)\overline{\sigma} = \overline{\sigma} \neq 0$. The image of $\overline{\sigma}$ in $H_1(G_g,\ZZ)$ is therefore nonzero. Finally, since $\sigma$ acts trivially on its set of fixed points, its image under the composition
$$\Delta_g\rightarrow G_g\rightarrow S_{2g+2}$$
is trivial. Thus when $g$ is odd, the map $H_1(G_g, \ZZ)\rightarrow \ZZ/2\ZZ$ has nontrivial kernel. Therefore, $H_1(G_g,\ZZ)\cong \ZZ/4\ZZ$ for $g$ odd.
\end{proof}
\subsection{The Deligne-Mumford Compactification of $\cHH_g$} 

Let $\overline{\cHH}_g$ denote the closure of $\cHH_g$ in $\overline{\cM}_g$. This is the moduli space of stable hyperelliptic curves of genus $g$. These are stable curves that arise as the limit of a family of smooth hyperelliptic curves. It can be shown that $\overline{\cHH}_g$ is globally the quotient of a smooth variety by a finite group. In fact, $\overline{\cHH}_g = S_{2g+2}\backslash\overline{\cM}_{0,2g+2}$, where $S_{2g+2}$ acts on $\overline{\cM}_{0,2g+2}$ by permuting the marked points.

\vspace{.1in}
For our work, we will need to have some description of the boundary divisor 
$$\partial \cHH_g = \overline{\cHH}_g\setminus \cHH_g.$$ Note that $\partial \cHH_g$ is the intersection of $\overline{\cHH}_g$ with $\partial \cM_g$, the latter of which has irreducible decomposition $\Delta_{irr}\cup D_1\cup \cdots \cup D_{[g/2]}$.  

First, the intersection of $\overline{\cHH}_g$ with $\Delta_{irr}$ is the union of $\displaystyle [(g-1)/2]$ irreducible divisors $E_0,\ldots, E_{ \floor{(g-1)/2}}$, where a generic point of $E_j$ represents an irreducible hyperelliptic curve formed by joining a smooth  hyperelliptic curve of genus $i$ to a smooth hyperelliptic curve of genus $g-i-1$ at two points $\{p_1,p_2\}$ that are exchanged by the hyperelliptic involution \cite{acgh}. The intersection of $\partial \cHH_g$ with the divisor $D_i$ consists of a single irreducible component $\delta_j$, whose generic point represents a curve formed by joining a smooth curve of genus $i$ to a smooth curve of genus $g-i$ at a Weierstrass point. 

\vspace{.1in}
From the description of $\partial \cHH_g$ given above, it is clear that 
$$\cHH_g^c = \overline{\cHH}_g\setminus E_0\cup \cdots \cup E_{[(g-1)/2]}.$$

It is shown in \cite{acgh} that 
$$\text{Pic}\ \overline{\cHH}_g\otimes \QQ \cong \QQ^{g}$$
and that $\text{Pic}\ \overline{\cHH}_g$ is freely generated by the classes of the $g$ divisors $E_0,\ldots, E_{[(g-1)/2]}$ and $D_1,\ldots, D_{[g/2]}$.

\vspace{.1in}
Since $\overline{\cHH}_g$ is has only finite quotient singularities, we have an exact sequence
$$0\rightarrow \bigoplus_{j=0}^{\floor{(g-1)/2}} \QQ E_j \rightarrow \text{Pic}\ \overline{\cHH}_g\otimes\QQ\rightarrow \text{Pic}\ \cHH_g^c\otimes \QQ\rightarrow 0.$$
Therefore $\text{Pic}\ \cHH_g^c\ \otimes \QQ = \QQ^{[g/2]}$. 

\vspace{.1in}
We are now have everything we need to prove Theorem \ref{picardhyperellipticcompacttype}.
\begin{proof}[Proof of Theorem \ref{picardhyperellipticcompacttype}]
By Theorem \ref{hainpic0theorem} and Theorem \ref{putmanpicardtheorem}, $\text{Pic}\ \cHH_g^c$ is finitely generated with torsion subgroup isomorphic to $H_1(G_g,\ZZ)$. Proposition \ref{abelianizationGg} shows that this latter group is isomorphic to $\ZZ/2\ZZ$ when $g$ is even and $\ZZ/4\ZZ$ when $g$ is odd. Since, by the discussion above, $\text{Pic}\ \cHH_g^c \otimes \QQ \cong \QQ^{\floor{g/2}}$, the free part of $\text{Pic}\ \cHH_g^c$ is abstractly isomorphic to $\ZZ^{[g/2]}$. This concludes the proof.
\end{proof}

\subsection{Review of Mixed Hodge Theory}
In order to analyze $\text{Pic}\ \cHH_g[m]$, we call on certain results from the mixed Hodge theory. We quickly review  some basic results on mixed Hodge structures here.    
\subsubsection{Mixed Hodge structures}
Let $H$ be an abelian group. A pure Hodge structure of weight $k\geq 0$ on $H$ is a descending filtration $F^{\bullet}$ of $H\otimes \CC$ such that we have a direct sum decomposition
$$H\otimes \CC = \bigoplus_{p+q = k}H^{p,q}$$
where $H^{p,q} = F^p\cap \overline{F^{q}}$. The filtration $F^{\bullet}$ is called the Hodge filtration. If $H\otimes \QQ$ is equipped with an increasing filtration $W_{\bullet}$ with the property that the Hodge filtration $F^{\bullet}$ induces a pure Hodge structure of weight $l$ on $\displaystyle \text{Gr}_l^WH\otimes \CC$, then we say that $H$ is a mixed Hodge structure (MHS).

The category of mixed Hodge structures is abelian. Deligne has shown that for each $k\geq 0$ the cohomology $H^k(X)$ of a complex algebraic variety admits a functorial mixed Hodge structure. For a smooth projective variety $X$ this Hodge structure is pure of weight $k$. If $X$ is smooth but not complete, the mixed Hodge structure on $H^k(X)$ will not, in general, be pure. Instead, the weight-graded quotients $\text{Gr}_l^WH^k(X,\QQ)$ may be non-zero for $l\in [k,2k]$.

\subsubsection{Poincar\'{e} duality}
On a smooth variety $X$ of dimension $d$, Poincar\'{e} duality gives an isomorphism of MHS $H_c^k(X)\cong (H^{2d-k}(X))^*(-d)$.

\subsubsection{A long exact sequence}
The compactly supported cohomology of a complex variety has a canonical mixed Hodge structure. If $U$ is a Zariski closed subset of $X$, there is a long exact sequence of compactly supported cohomology \cite{peterssteenbrink} where all maps are morphisms of MHS
$$\cdots \rightarrow H^k_c(X-U)\rightarrow H_c^k(X)\rightarrow H^k_c(U)\rightarrow H_c^{k+1}(X-U)\rightarrow \cdots$$

\subsection{Level covers of the hyperelliptic locus}
In this section, show that the Picard group of $\cHH_g[m]$ is finitely generated for all $m\geq 2$. When $m\geq 3$, $\cHH_g[m]$ is a smooth variety, essentially because $\Sp_g(\ZZ)[m]$ is torsion-free in these cases. Arguments from \cite{hain94} can then be applied directly. When $m = 2$, a slightly different approach is required, as $\cHH_g[2]$ has a non-trivial orbifold structure; this stems from the fact that $\Sp_g(\ZZ)[2]$ contains the torsion element $-\text{Id}$. 

\vspace{.1in}
\subsubsection{The case $m\geq 3$}
The complement of $\cHH_g[m]$ inside of $\cHH_g^c[m]$ is a divisor, which we denote by $D$. The codimension in $\cHH_g^c[m]$ of its singular locus $D^{sing}$ is equal to at least 2. We define
$$D^*= D-D^{sing}\ \ \ \ \text{and}\ \ \ \ \ \widetilde{\cHH_g^{c}}[m] = \cHH_g^c[m] - D^{sing}.$$
Notice that $D^*$  is Zariksi closed in $\widetilde{\cHH_g^{c}}[m]$. 

\vspace{.1in} Since $\widetilde{\cHH_g^{c}}[m]$ is Zariski open in $\cHH_g^{c}[m]$, each cohomology group $H^k(\widetilde{\cHH_g^{c}}[m],\ZZ)$ has a canonical mixed Hodge structure. When $k\leq 2$, this MHS coincides with the one on $H^k(\cHH_g^{c}[m],\ZZ)$ because the inclusion $\widetilde{\cHH_g^{c}}[m]\rightarrow \cHH_g^{c}[m] $ induces isomorphisms 
$$H^k(\widetilde{\cHH_g^{c}}[m],\ZZ )\rightarrow H^k(\cHH_g^{c}[m],\ZZ)$$ 
for $k \leq 2$ for reasons of codimension.
%
%
%
\begin{proposition}\label{hodgestructurewt2}
For each $m\geq 3$, the mixed Hodge structure on $H^1(\cHH_g[m],\ZZ)$ is pure of weight 2.
\end{proposition}
\begin{proof}
The long exact sequence of compactly supported cohomology with $X = \widetilde{\cHH_g^{c}}[m]$ and $U = D^*$ has a segment
\begin{equation}\label{LESMHS}
\cdots \rightarrow H^{2n-2}_c(D^*,\QQ)\rightarrow H_c^{2n-1}(\cHH_g[m], \QQ)\rightarrow H^{2n-1}_c(X,\QQ)\rightarrow  \cdots 
\end{equation}
By Poincar\'{e} duality, there are isomorphisms 
\begin{equation*}
H^{2n-2}_c(D^*,\QQ)\cong \displaystyle \bigoplus_{|D^*|}\QQ(-n+1)\ \ \ \ \text{and}\ \ \ \  H^{2n-1}_c(\cHH_g[m],\QQ)\cong  H^1(\cHH_g[m],\QQ)^*(-n)
\end{equation*}
where the first direct sum is indexed by the irreducible components of $D^*$. Note that 
$$H^{2n-1}_c(X)\cong H^1(X,\QQ)^*(-n) = 0$$ by equation (\ref{H1hyplocuslevel}). Since (\ref{LESMHS}) is an exact sequence of MHS, we have that $H^1(\cHH_g[m],\QQ)^*(-n)$ is pure of weight $2n-2$. In other words, $H^1(\cHH_g[m],\QQ)$ is pure of weight $2$.
\end{proof}

\begin{proposition}[Hain \cite{hain94}]\label{hainpic0proposition}
Suppose that $X$ is a smooth quasiprojective variety and that $\text{Gr}^W_1H^1(X,\QQ) = 0$. Then $\text{Pic}^0X = 0$ and $\text{Pic}\ X$ is finitely generated.
\end{proposition}

We now have everything we need to prove Theorem \ref{hyperelliptictheorem}
\begin{proof}[Proof of Theorem \ref{hyperelliptictheorem}]
Proposition \ref{hodgestructurewt2} implies that $\text{Gr}^W_1H^1(\cHH_g[m],\QQ) = 0$. An application of Proposition \ref{hainpic0proposition} shows immediately that $\text{Pic}\ \cHH_g[m]$ is finitely generated.
\end{proof}

\vspace{.1in}
\subsubsection{The case $m=2$}
At the level of varieties, there is an isomorphism $\cHH_g[2]\cong \cM_{0,2g+2}$. Thus the rational Picard group of the orbifold $\cHH_g[2]$ is isomorphic to $\text{Pic}\ \cM_{0,2g+2}\otimes \QQ = 0$. An application of Lemma \ref{picardlemma} shows that that $\text{Pic}\ \cHH_g[2]$ is finite; we now show that it is non-trivial.

\vspace{.1in}
The approach we shall take involves showing that the abelianization of $\Delta_g[2]$ has a non-trivial torsion subgroup. In order to do this, we introduce the \emph{symplectic Lie algebra}
$$\mathfrak{sp}_g(\FF_2) = \{ A\in \text{M}_{2g\times 2g}(\FF_2)\ |\ A^TJ + JA = 0\}\ \ \ \ \text{where}\ \ \ \ J = \left(\begin{array}{cc}0 & -I_{g\times g} \\I_{g\times g} & 0\end{array}\right).$$

It can be shown that $\mathfrak{sp}_g(\FF_2)\cong \FF_2^{g(2g+1)}$. By \cite{brendlemargalitlevel4} and \cite{putmanpicard} there is an isomorphism 
$$\Sp_g(\ZZ)[2]/\Sp_g(\ZZ)[4] \cong \mathfrak{sp}_g(\ZZ/2\ZZ).$$
By Theorem \ref{acampostheorem}, the natural maps
$$\Delta_g[m]\rightarrow \Sp_g(\ZZ)[m]$$
are surjective for $m$ even. This implies that 
$$\Delta_g[m_1]/\Delta_g[m_2] \cong \Sp_g(\ZZ)[m_1]/\Sp_g(\ZZ)[m_2]$$
whenever $m_2|m_1$ and both are even. We therefore have a short exact sequence 
\begin{equation}
1\rightarrow \Delta_g[4]\rightarrow \Delta_g[2]\rightarrow \mathfrak{sp}_g(\FF_2)\rightarrow 1.
\end{equation}

\begin{lemma}\label{mod2torsionlemma}
For each $g\geq 2$, there is an isomorphism 
$$H_1(\Delta_g[2], \ZZ)\cong \ZZ^{g(2g+1)-1}\oplus \ZZ/2\ZZ $$
\end{lemma}
\begin{proof}
For each $g\geq 2$, there is a short exact sequence
\begin{equation}\label{level2extension}
1\rightarrow \langle \sigma \rangle \rightarrow \Delta_g[2]\rightarrow \text{PMod}(S_{0,2g+2})\rightarrow 1
\end{equation}
Recall that the abelianization of $\text{PMod}(S_{0,2g+2})$ is free abelian of rank $g(2g+1)-1$. The 5-term exact sequence associated to (\ref{level2extension}), has a segment
\begin{equation}\label{deltag25term}
\langle \sigma \rangle \rightarrow H_1(\Delta_g[2] ,\ZZ)\rightarrow \ZZ^{g(2g+1)-1}\rightarrow 0.
\end{equation}
We will show that the image of $\langle \sigma\rangle \rightarrow H_1(\Delta_g[2] ,\ZZ)$ is non-trivial, hence an embedding. 
 
The induced homomorphism $\Delta_g[2]\rightarrow \mathfrak{sp}_g(\FF_2)$ factors through $H_1(\Delta_g[2],\ZZ/2\ZZ)$ to give a surjective homomorphism 
$$H_1(\Delta_g[2], \ZZ/2\ZZ)\rightarrow \mathfrak{sp}_g(\FF_2) \cong (\ZZ/2\ZZ)^{g(2g+1)}.$$

Since $H_1(\Delta_g[2],\ZZ/2\ZZ) = H_1(\Delta_g[2], \ZZ)\otimes \ZZ/2\ZZ$, a dimension count shows that  $\langle \sigma\rangle$ is contained in $H_1(\Delta_g[2],\ZZ)$. Since (\ref{deltag25term}) splits, the result follows.
\end{proof}

\begin{proposition}
For all $g\geq 2$, the first Chern class induces a surjective homomorphism 
$$\text{Pic}\ \cHH_g[2]\rightarrow \ZZ/2\ZZ.$$
\end{proposition}
\begin{proof}
The image of the first Chern class homomorphism
$$c_1: \text{Pic}\ \cHH_g[2]\rightarrow H^2(\Delta_g[2],\ZZ)$$
is isomorphic to the torsion subgroup of $H_1(\Delta_g[2],\ZZ)$. By Lemma \ref{mod2torsionlemma}, the torsion subgroup of $H_1(\Delta_g[2],\ZZ)$ is isomorphic to $\ZZ/2\ZZ$. 
\end{proof}

\begin{remark}
It is currently unclear whether $\text{Pic}\ \cHH_g[m]$ is non-trivial when $m\geq 3$. The torsion subgroup of $\text{Pic}\ \cHH_g[m]$ is equal to the torsion subgroup of $H_1(\Delta_g[m],\ZZ)$; it is an open problem to determine whether or not this group is trivial. 
\end{remark}

\end{document}